\documentclass[11pt]{article}
\usepackage[utf8]{inputenc}

\usepackage[left=23mm,right=23mm,top=23mm,bottom=23mm,paper=a4paper]{geometry}

\usepackage{graphicx} 
\usepackage{amssymb, amsmath, amsfonts, amsthm, graphics}
\usepackage{subfig}
\usepackage{multicol}
\usepackage{xcolor}
\usepackage[colorlinks = true,
            linkcolor = blue,
            urlcolor  = blue,
            citecolor = blue,
            anchorcolor = blue]{hyperref}
\usepackage{cleveref}
\usepackage{wrapfig}
\usepackage[square,numbers,sort&compress]{natbib}
\usepackage[font=small,labelfont=bf]{caption}
\usepackage{booktabs}

\usepackage{algorithm}
\usepackage[noend]{algorithmic}

\newcommand{\smtt}[1]{\texttt{\small #1}}
\newcommand{\smsf}[1]{\textsf{\small #1}}

\newcommand{\algocomment}[1]{{\hfill \smtt{/\!/ #1\quad}}}

\theoremstyle{plain}
\newtheorem{theorem}{Theorem}[section] 
\newtheorem{lemma}[theorem]{Lemma}
\newtheorem{fact}[theorem]{Fact}

\newcommand{\eat}[1]{}

\newcommand{\EE}{\mathbb{E}}

\newcommand{\RR}{\mathbb{R}}
\renewcommand{\SS}{\mathbb{S}}

\newcommand{\cA}{\mathcal{A}}

\newcommand{\cC}{\mathcal{C}}

\newcommand{\cF}{\mathcal{F}}

\newcommand{\cI}{\mathcal{I}}

\newcommand{\cK}{\mathcal{K}}
\newcommand{\cL}{\mathcal{L}}

\newcommand{\cN}{\mathcal{N}}
\newcommand{\cO}{\mathcal{O}}

\newcommand{\cX}{\mathcal{X}}
\newcommand{\cY}{\mathcal{Y}}
\newcommand{\cZ}{\mathcal{Z}}

\DeclareMathOperator*{\argmax}{arg\,max}
\DeclareMathOperator*{\argmin}{arg\,min}
\DeclareMathOperator*{\minimize}{min}
\DeclareMathOperator*{\maximize}{max}
\DeclareMathOperator*{\st}{s.\!t.\!}
\newcommand{\proj}{\mathrm{proj}}
\newcommand{\dist}{\mathrm{dist}}

\newcommand{\adj}{\ensuremath{\mathclose{\vphantom{)}}^*}}
\newcommand{\nuc}{\ensuremath{\mathclose{\vphantom{)}}_*}}
\newcommand{\fro}{\ensuremath{\mathclose{\vphantom{)}}_\textrm{F}}}
\newcommand{\op}{\ensuremath{\mathclose{\vphantom{)}}_\textrm{op}}}

\newcommand{\econstraint}{$\exists$-constraint\;}
\newcommand{\econstraints}{$\exists$-constraints\;}

\newcommand{\svec}{\mathrm{svec}}
\newcommand{\rank}{\mathrm{rank}}

\newcommand{\tr}{\mathrm{tr}}

\title{Fast, Scalable, Warm-Start Semidefinite Programming \\ with Spectral Bundling and Sketching}
\author{Rico Angell\thanks{Manning College of Information and Computer Sciences, University of Massachusetts Amherst, Amherst, MA, USA. Correspondence to Rico Angell (\href{mailto:rangell@cs.umass.edu}{rangell@cs.umass.edu}).} \and  Andrew McCallum\footnotemark[1]}
\date{}

\begin{document}

\maketitle

\begin{abstract}
\noindent
While semidefinite programming (SDP) has traditionally been limited to moderate-sized problems, recent algorithms augmented with matrix sketching techniques have enabled solving larger SDPs.
However, these methods achieve scalability at the cost of an increase in the number of necessary iterations, resulting in slower convergence as the problem size grows.
Furthermore, they require iteration-dependent parameter schedules that prohibit effective utilization of warm-start initializations important in practical applications with incrementally-arriving data or mixed-integer programming.
We present Unified Spectral Bundling with Sketching (USBS), a provably correct, fast and scalable algorithm for solving massive SDPs that can leverage a warm-start initialization to further accelerate convergence.
Our proposed algorithm is a spectral bundle method for solving general SDPs containing both equality and inequality constraints.  
Moveover, when augmented with an optional matrix sketching technique, our algorithm achieves the dramatically improved scalability of previous work while sustaining convergence speed.
We empirically demonstrate the effectiveness of our method across multiple applications, with and without warm-starting. For example, USBS provides a 500x speed-up over the state-of-the-art scalable SDP solver on an instance with over 2 billion decision variables.
\end{abstract}

\section{Introduction}

Semidefinite programming (SDP) is a convex optimization paradigm capable of modeling or approximating many practical problems in combinatorial optimization~\cite{alizadeh1995interior}, neural network verification~\cite{raghunathan2018semidefinite,dathathri2020enabling}, robotics~\cite{rosen2019se}, optimal experiment design~\cite{vandenberghe1998determinant}, VLSI~\cite{vandenberghe1998optimizing}, and systems and control theory~\cite{bertsimas1995achievable}.
Despite their widespread applicability, practitioners often dismiss the use of SDPs under the presumption that optimization is intractable at real-world scale.
This assumption is grounded in the prohibitively high computational complexity of standard SDP solvers~\cite{alizadeh1998primal,vandenberghe1999applications}. 

The main challenge when solving SDPs with standard constrained optimization approaches is the high cost of projection onto the feasible region.
Projecting a symmetric matrix onto the semidefinite cone requires a full eigendecomposition, an operation that scales cubicly with the problem dimension.
The high computational cost of this single operation severely limits the applicability of projected gradient and ADMM-based approaches~\cite{boyd2011distributed,scs}.
Interior point methods, the de facto approach for solving small-to-moderate sized SDPs, require even more computational resources.

Projection-free and conditional gradient-based methods have been proposed to avoid the computational burden of projecting iterates onto the semidefinite cone~\cite{yurtsever2019conditional,yurtsever2015universal,arora2005fast}.
These methods avoid projection onto the feasible region with various techniques including subgradient descent in the dual space and a conditional gradient method on an unconstrained augmented objective function.
The bulk of the computational burden for these methods is typically an extreme eigenvector calculation.
The scalablity of these methods is still limited, despite the improved per-iteration complexity, since they require storing the entire primal matrix which scales quadratically with the problem dimension.

The most well-known method for scaling semidefinte programming without storing the entire primal matrix in memory is the Burer–Monteiro (BM) factorization heuristic~\cite{burer2003nonlinear}.
The core idea of the BM method is to explicitly control the memory usage by restricting the primal matrix to a low-rank factorization.
However, this method sacrifices the convexity of the problem for scalability. In some cases, the BM method either requires having a high rank factorization, leading to burdensome memory requirements, or risks optimization getting stuck in local minima~\cite{waldspurger2020rank}.
\citet{yurtsever2021scalable} extend the conditional gradient augmented Lagrangian approach~\cite{yurtsever2019conditional} with a matrix sketching technique. This approach avoids storing the entire primal matrix while maintaining the convexity of the problem leading to provable convergence for general SDPs.
In addition, this method uses iteration-dependent parameters prohibiting it from reliably leveraging a warm-start initialization.

Spectral bundle methods, first proposed by~\citet{helmberg2000spectral}, are an appealing framework for solving SDPs due to their low per-iteration computational complexity and fast empirical convergence. 
While several spectral bundle methods have been presented in the literature  ~\cite{helmberg2002spectral,apkarian2008trust,helmberg2014spectral,ding2023revisiting}, previous work considers SDPs with either only equality constraints or only inequality constraints.
Furthermore, the lack of an efficient standalone implementation has prevented the evaluation of spectral bundle method on massive SDPs.

\paragraph{Contributions.} We present USBS, a unified spectral bundle method designed for a broader class of SDPs and show it is a practical approach for solving large problem instances quickly. 
USBS flexibly allows the user to control the trade-off between per-iteration complexity and the empirical speed at which the algorithm converges.
In addition, USBS can be augmented with a matrix sketching technique that can dramatically improve the scalability of the algorithm while maintaining its fast convergence.
 We demonstrate the empirical efficacy of our proposed spectral bundle method on three types of practical large SDPs seeing extraordinary performance improvements in comparison with the previous state-of-the-art scalable solver for general SDPs.
For example, USBS is able to obtain a quality solution to an SDP with over $10^{13}$ decision variables while the previous state-of-the-art fails to reach an accurate solution within 72 hours.
In addition, we see a 500x speedup over the previous state-of-the-art on an instance with 2 billion decision variables.
Finally, we find that warm-starting can lead to more than a 100x speedup in the convergence of USBS as compared to cold-starting while the previous state-of-the-art often is not able to reliably take advantage of a warm-start initialization.
We provide a standalone implementation of the algorithm in pure JAX~\cite{bradbury2018jax, frostig2018compiling}, enabling efficient execution on various hardware (CPU, GPU, TPU) and wide-spread use.

\section{Preliminaries}


\subsection{Semidefinite Programming}
We begin with the notation necessary to define the semidefinite programming problem.
Let $\SS^n$ be the set of all real symmetric $n \times n$ matrices and $\SS
^n_+ := \{X \in \SS^n : X \succeq 0\}$ be the positive semidefinite cone. Let $\langle \cdot, \cdot \rangle$ denote the standard inner product for vectors and the Frobenius inner product for matrices.
For any matrix $M \in \SS^n$, let $\lambda_{\max}(M)$ denote the maximum eigenvalue of $M$ and let $\tr(M)$ denote the trace of $M$.
Let $\cA: \SS^n \to \RR^m$ be a given linear operator and $\cA\adj : \RR^m \to \SS^n$ be its adjoint, i.e. $\langle \cA X, y \rangle = \langle X, \cA\adj y \rangle$ for any $X \in \SS^n$ and $y \in \RR^m$, which take the following form,
\begin{equation*}
   \cA X = \begin{bmatrix} 
   \langle A_1, X \rangle \\
   \vdots \\
   \langle A_m, X \rangle \\
   \end{bmatrix}
    \quad \textrm{and} \quad 
    \cA\adj  y = \sum_{i=1}^m y_i A_i,
\end{equation*}
for given matrices $A_i \in \SS^n$. For any given $\cI \subset \{1,\ldots,m\}$, let $\cA_{\cI}$ denote the linear operator restricted to matrices $A_i$ for $i \in \cI$ and $y_{\cI}$ to be the corresponding subset of rows for any $y \in \RR^m$. In general, we utilize $(\cdot)_\cI$ to extract the corresponding indices from any vector in $\RR^m$. Finally, let $\cI':= \{1, \cdots, m\} \setminus \cI$ be the compliment of $\cI$. 

Given fixed $C \in \SS^n$, $\cA$, $b \in \RR^m$, and $\cI$, the primal~\eqref{eq:primal_prob} and dual~\eqref{eq:dual_prob} semidefinite programs can be written as follows:
\vspace*{-0.75cm}
\begin{multicols}{2}
\begin{equation*}
    \begin{aligned}
        &\!\!\maximize_{X \succeq 0} \;\langle C, X \rangle \\
        &\st \;\; \cA_{\cI^{'}} \!X = b_{\cI^{'}} \\
        &\qquad\cA_{\cI} X \leq b_{\cI}
    \end{aligned}\!\!\!\tag{P}\label{eq:primal_prob}
\end{equation*}
\vfill
  \columnbreak
\begin{equation*}
    \begin{aligned}
        &\!\!\minimize_{y \in \RR^m} \; \langle b, y \rangle \\
        &\st \;\; C - \cA\adj  y \preceq 0 \\
        &\qquad y_{\cI} \geq 0
    \end{aligned}\!\!\!\!\!\tag{D}\label{eq:dual_prob}
\end{equation*}
\end{multicols}
\vspace*{-0.25cm}
Note that while all SDPs of this form can technically be written as equality constrained SDPs (i.e. \emph{standard-form} SDPs), doing so in practice is ill-advised since the problem dimension, $n$, grows by one for every inequality constraint.
In addition, we also represent the primal feasibility set as the convex set $\cK := \{ z \in \RR^m : z_\cI \leq b_\cI, z_{\cI'} = b_{\cI'} \}$,
and measure the primal infeasibility $\dist(\cA X, \cK)$. 
We use $\dist(z, \cZ)$ to denote the Euclidean distance from a point $z$ to a closed set $\cZ$ and $\proj_\cZ(z)$ to denote the Euclidean projection of $z$ onto $\cZ$.
Equivalently, 
\begin{equation*}
    \dist(z, \cZ) = \inf_{z' \in \cZ} \| z - z' \| = \|z - \proj_\cZ(z) \|.
\end{equation*}
We denote the solution sets of~\eqref{eq:primal_prob} and~\eqref{eq:dual_prob} by $\cX_\star$ and $\cY_\star$, respectively.
We make the standard assumption that~\eqref{eq:primal_prob} and~\eqref{eq:dual_prob} satisfy \emph{strong duality}, namely that the solution sets $\cX_\star$ and $\cY_\star$ are nonempty, compact, and every pair of solutions $(X_\star, y_\star) \in \cX_\star \times \cY_\star$ has zero duality gap: $p_\star := \langle C, X_\star \rangle = \langle b, y_\star \rangle =: d_\star$.
Strong duality holds whenever Slater's condition holds and $\cA$ is a surjective linear operator.
Additionally, we say that the SDP satisfies \emph{strict complementarity} if there exists $(X_\star, y_\star)$ such that for induced dual slack matrix $Z_\star := C - \cA\adj y_\star$, $ \rank(X_\star) + \rank(Z_\star) = n.$
Strict complementarity is satisfied by generic SDPs~\cite{alizadeh1997complementarity} and many well-structured SDPs~\cite{ding2021simplicity}.
Lastly, we denote the maximum nuclear norm of the primal solution set as $\mathfrak{N}(\cX_\star) := \sup_{X_\star \in \cX_\star} \| X_\star\|\nuc$.

In many important applications, SDPs take a highly structured and sparse forms that admit low-rank solutions.
The cost matrix $C$ is often sparse, containing many fewer than $n^2$ non-zero entries.
Additionally, the number of primal constraints $m$ is often much less $n^2$, being proportional to $n$ or the number of non-zero entries in $C$.
We would like to exploit these features with a provably correct, fast and scalable optimization algorithm that can solve any SDP.


\subsection{Proximal Bundle Method}

Proximal bundle methods~\cite{lemarechal1981bundle,mifflin1977algorithm,feltenmark2000dual,kiwiel2000efficiency} are a class of optimization algorithms for solving unconstrained convex minimization problems of the form $\min_{y \in \RR^m} f(y)$,
where $f: \RR^m \to (-\infty, +\infty]$ is a proper closed convex function that attains its minimum value, $\inf f$, on some nonempty set. At each iteration, the proximal bundle method proposes an update to the current iterate $y_t$ by applying a proximal step to an approximation of the objective function, $\hat{f}_t$:
\begin{equation}
\label{eq:proposal_step}
\tilde{y}_{t+1} \gets \argmin_{y \in \RR^m} \hat{f}_t(y) + \frac{\rho}{2} \| y - y_t \|^2
\end{equation}
where $\rho > 0$.
The next iterate $y_{t+1}$ is set equal to $\tilde{y}_{t+1}$ only when the decrease in the objective value is at least a fixed fraction of the decrease predicted by the model $\hat{f}_t$, i.e.
\begin{equation}
\label{eq:descent_cond}
\beta (f(y_t) - \hat{f}_t(\tilde{y}_{t+1})) \leq f(y_t) - f(\tilde{y}_{t+1})
\end{equation}
for some fixed $\beta \in (0, 1)$. The iterations where~\eqref{eq:descent_cond} is satisfied (and thus, $y_{t+1} \gets \tilde{y}_{t+1}$) are referred to as \emph{descent steps}. Otherwise, the algorithm takes a \emph{null step} and sets $y_{t+1} \gets y_{t}$. 
Regardless of whether~\eqref{eq:descent_cond} is satisfied or not, $\tilde{y}_{t+1}$ is used to construct the next model $\hat{f}_{t+1}$.

\paragraph{Model Requirements.} The model $\hat{f}_t$ can take many forms, but is usually constructed using subgradients of $f$ at past and current iterates. Let $\partial f(y) := \{g : f(y') \geq f(y) + \langle g, y' - y \rangle, \,\forall y' \in \RR^m\}$ denote the subdifferential of $f$ at $y$ (i.e.\!\! the set of subgradients of $f$ evaluated at a point $y$). Following prior work, we require the next model $\hat{f}_{t+1}$ to satisfy the following mild assumptions:
\begin{enumerate}
    \item \textit{\textbf{Minorant.}}
    \begin{equation}
    \label{eq:minorant_cond}
        \hat{f}_{t+1}(y) \leq f(y),\, \forall y \in \RR^m
    \end{equation}
    \item \textit{\textbf{Subgradient lower bound.}} For any $ g_{t+1} \in \partial f(\tilde{y}_{t+1})$, 
    \begin{equation}
        \label{eq:obj_subgrad_cond}
        \hat{f}_{t+1}(y) \geq f(\tilde{y}_{t+1}) + \langle g_{t+1}, y - \tilde{y}_{t+1}\rangle,\, \forall y \in \RR^m
    \end{equation}
    \item \textit{\textbf{Model subgradient lower bound.}} The first order optimality conditions for~\eqref{eq:proposal_step} gives the subgradient $s_{t+1} := \rho(y_t - \tilde{y}_{t+1}) \in \partial \hat{f}_{t} (\tilde{y}_{t+1})$. After a null step $t$, 
    \begin{equation}
        \label{eq:model_subgrad_cond}
        \hat{f}_{t+1}(y) \geq \hat{f}_{t}(\tilde{y}_{t+1}) + \langle s_{t+1}, y - \tilde{y}_{t+1}\rangle,\, \forall y \in \RR^m
    \end{equation}
\end{enumerate}
\vspace{-1em}
The first two conditions serve to guarantee that a new model integrates first-order information from the objective at $\tilde{y}_{t+1}$. The third condition demands the new model to preserve approximation accuracy exhibited by the preceding model.

Assuming these model requirements, \citet{diaz2023optimal} proved non-asymptotic convergence rates for the proximal bundle method under various conditions.
We summarize the results relevant to this work in Theorem~\ref{thm:proximal_bundle_convergence}.

\section{Unified Spectral Bundling}
\label{sec:specbm}
In this section, we will present USBS, our proposed algorithm for solving the SDP defined in~\eqref{eq:primal_prob} and~\eqref{eq:dual_prob}, 
Our proposed spectral bundle method is a unified algorithm for solving SDPs with both equality and inequality constraints and can be integrated with a matrix sketching technique for dramatically improved scalability as demonstrated in Section~\ref{sec:sketching}.
To apply the proximal bundle method, we consider consider minimizing the following unconstrained objective 
\begin{equation}
    f(y) := \alpha \, [\lambda_\textrm{max}(C - \cA\adj y)]_+ + \langle b, y \rangle + \iota_\cY(y),\tag{pen-D}\label{eq:pen_dual}
\end{equation}
where $[\,\cdot\,]_+ := \max\{\,\cdot\,, 0\}$ and $\cY := \{y \in \RR^m : y_{\cI} \geq  0\}$.
Minimizing~\eqref{eq:pen_dual} is equivalent to optimizing~\eqref{eq:dual_prob} in the sense that the optimal solution set and objective are the same (Appendix~\ref{sec:derive_pen_dual}).
In the following subsections, we will detail how the model is constructed and how to compute the candidate iterates.

\subsection{Spectral Bundle Model}
The form of~\eqref{eq:pen_dual} is not quite conducive to defining the model and solving for the candidate iterate $\tilde{y}_{t+1}$.
To address this, we will consider an equivalent formulation to~\eqref{eq:pen_dual} that is more amenable to this effort.
For any fixed $\alpha \geq 2 \mathfrak{N}(\cX_\star)$, let $\cX := \{X \in \SS^n_+ : \tr(X) \leq \alpha\}$ be the trace constrained subset of the primal domain. 
Let $\mathrm{N} := \{\nu \in \RR^m : \nu_{\cI} \leq 0, \nu_{\cI^{'}} = 0\}$ be the domain of the dual slack variable $\nu \in \RR^m$ for the indicator function $\iota_\cY(\cdot)$.
Then, we can rewrite the following equivalent formulation to~\eqref{eq:pen_dual} 
\begin{equation}
    \label{eq:equiv_pen_dual}
    f(y) = \sup_{(X, \nu) \in \cX \times \mathrm{N}} \langle C - \cA\adj y , X\rangle + \langle b - \nu, y \rangle.
\end{equation}
We derive this equivalent formulation in Appendix~\ref{sec:equiv_derive_pen_dual}.

\paragraph{Defining the model.} This equivalent formulation is clearly just as difficult to optimize as~\eqref{eq:pen_dual}, but is helpful in defining the model we will utilize in the spectral bundle method.
The main idea is to consider~\eqref{eq:equiv_pen_dual} over a low-dimensional subspace of $\cX$ such that ~\eqref{eq:minorant_cond},~\eqref{eq:obj_subgrad_cond}, and \eqref{eq:model_subgrad_cond} are satisfied.
The subspace we will consider at step $t$ is parameterized by matrices $\bar{X}_t \in \SS^n_+$ and $V_t \in \RR^{n \times k}$ and is defined as follows
\begin{equation}
\label{eq:subspace_def}
\widehat{\cX}_t := \left\{\eta \bar{X}_t + V_t S V_t^\top \middle\vert \begin{array}{l}\eta\, \tr(\bar{X}_t) + \tr(S) \leq \alpha \\ \eta \geq 0, S \in \SS^k_+\end{array}\!\!\!\right\}.
\end{equation}
The matrix $V_t$ has $k$ orthonormal column vectors, where $k$ is a small user defined parameter.
The columns of $V_t$ are partitioned into $k_c \geq 1$ current eigenvectors and $k_p \geq 0$ orthonormal vectors representing past spectral information.
Regardless of the setting of $k_c$ and $k_p$, the columns of $V_t$ will always include a maximum eigenvector $v_1$ of $C - \cA\adj  \tilde{y}_t$, which also means $b - \alpha\cA v_1 v_1^\top \in \partial f(\tilde{y}_t)$.
For scalability reasons, we expect $k = k_c + k_p$ to be relatively small.
The matrix $\bar{X}_t$ is a carefully selected weighted sum of past spectral bounds such that $\tr(\bar{X}_t) \leq \alpha$ and enables the last model condition \eqref{eq:model_subgrad_cond} to be satisfied.
Hence, the model is
\begin{equation}
    \label{eq:model_def}
    \hat{f}_t(y) := \sup_{(X, \nu) \in \widehat{\cX}_t \times \mathrm{N}} \langle C - \cA\adj y , X\rangle + \langle b - \nu, y \rangle.
\end{equation}
We show that this model satisfies the conditions~\eqref{eq:minorant_cond},~\eqref{eq:obj_subgrad_cond}, and \eqref{eq:model_subgrad_cond} in Appendix~\ref{sec:sublinear_convergence_proof}.

\paragraph{Updating the model.} Independent of whether a descent step or null step is taken, the model needs to be updated.
At every step, we solve the following minimax problem
\begin{equation}
    \min_{y \in \RR^m} \sup_{(X, \nu) \in \widehat{\cX}_t \times \mathrm{N}} \langle C - \cA\adj y , X\rangle + \langle b - \nu, y \rangle + \frac{\rho}{2} \| y - y_t \|^2,
\end{equation}
where $(\tilde{y}_{t+1}, X_{t+1}, \nu_{t+1})$ is the minimax solution.
The structure of $\widehat{\cX}_t$ allows us to rewrite $X_{t+1}$ as
\begin{equation}
    X_{t+1} = \eta_{t+1} \bar{X}_t + V_t S_{t+1} V_t^\top.
\end{equation}
To compute $\bar{X}_{t+1}$ and $V_{t+1}$ we first compute an
eigendecomposition of $S_{t+1}$ to separate current and past spectral information as follows
\begin{equation}
    \label{eq:model_update_eigendecomposition}
    S_{t+1} = Q_{\overline{p}} \Lambda_{\overline{p}} 
 Q_{\overline{p}}^\top + Q_{\underline{c}} \Lambda_{\underline{c}} 
 Q_{\underline{c}}^\top,
\end{equation}
where $\Lambda_{\overline{p}}$ is a diagonal matrix containing the largest $k_p$ eigenvalues and $Q_{\overline{p}} \in \RR^{k \times k_p}$ contains the corresponding eigenvectors.
The diagonal matrix $\Lambda_{\underline{c}}$ contains the remaining $k_c$ eigenvalues and the corresponding eigenvectors are contained in the columns of $Q_{\underline{c}} \in \RR^{k \times k_c}$.
Given this decomposition, we set the next model's $\bar{X}_{t+1}$ as follows
\begin{equation}
    \label{eq:x_bar_def}
    \bar{X}_{t+1} \gets \eta_{t+1} \bar{X}_t + V_t Q_{\underline{c}} \Lambda_{\underline{c}} Q_{\underline{c}}^\top V_t^\top.
\end{equation}
In the case that $k_p = 0$, observe that $\bar{X}_t = X_t$ for all $t$.
To update $V_{t+1}$, we first compute the top $k_c$ orthonormal eigenvectors $v_1, \ldots, v_{k_c}$ of $C - \cA\adj \tilde{y}_{t+1}$.
Then, we set $V_{t+1}$ to $k$ orthonormal vectors that span the columns of $[V_tQ_{\overline{p}}; v_1, \ldots, v_{k_c}]$ as follows
\begin{equation}
    \label{eq:v_next_def}
    V_{t+1} \gets \!\!\! \smtt{ orthonormalize}\Big([V_tQ_{\overline{p}}\,; v_1, \ldots, v_{k_c}]\Big),
\end{equation}
where we can use QR decomposition to orthonormalize the vectors. If $k_p = 0$, orthonormalization is unnecessary and we just set $V_{t+1} \gets [v_1, \ldots, v_{k_c}]$.

\subsection{Computing the Candidate Iterate} 

We will now discuss how to compute the candidate iterate $\tilde{y}_{t+1}$ by applying a proximal step to the current approximation of the objective function (i.e. solve~\eqref{eq:proposal_step}).
We show in Appendix~\ref{sec:derive_cand_iterate} that the candidate iterate can be computed as
\begin{equation}
    \label{eq:dual_update}
    \tilde{y}_{t+1} = y_t - \frac{1}{\rho} (b - \nu_{t+1} - \cA X_{t+1}),
\end{equation}
where $(X_{t+1}, \nu_{t+1}) \in \widehat{\cX}_t \times \mathrm{N}$ is the solution to the following optimization problem
\begin{equation}
(X_{t+1}, \nu_{t+1}) \in \argmax_{(X, \nu) \in \widehat{\cX}_t \times \mathrm{N}} \psi_t(X, \nu),
\end{equation}
where we define $\psi_t(X, \nu) :=  \langle C, X \rangle + \langle b - \nu - \cA X , y_t\rangle - \frac{1}{2\rho} \| b - \nu - \cA X \|^2.$ To solve for $(X_{t+1}, \nu_{t+1}) \in \widehat{\cX}_t \times \mathrm{N}$, we propose using the alternating maximization algorithm. After initializing $\tilde{\nu} = 0$, the the following update steps are repeated until convergence:
\begin{align}
    \tilde{X} &\gets \argmax_{X \in \widehat{\cX}_t} \; \psi_t(X, \tilde{\nu}) \label{eq:small_quad_sdp}\\
    \tilde{\nu} &\gets \argmax_{\nu \in \mathrm{N}} \; \psi_t(\tilde{X}, \nu) = \proj_{\mathrm{N}}(b - \cA\tilde{X} - \rho y_t) \label{eq:dual_slack_update}
\end{align}
Clearly $\psi_t$ is smooth, and in this case the alternating maximization algorithm is known to converge at a $\cO(1/\varepsilon)$ rate~\cite{beck2015convergence}.
If $\psi_t$ happens to also be strongly concave, the rate improves to $\cO(\log(1/\varepsilon))$~\cite{Luo1993ErrorBA}.
We solve~\eqref{eq:small_quad_sdp} using a primal-dual interior point method derived in Appendix~\ref{sec:small_quad_sdp_ipm}.

\begin{algorithm}[t!]
   \caption{Unified Spectral Bundling}
   \label{alg:spectral_bundle}
\begin{algorithmic}[1]
   \STATE \textbf{Input:} Problem specification $(C, \cA, b, \cI)$, parameters $(\alpha, \rho, \beta, k_c, k_p)$,
   and initialization of $X_0 = \bar{X}_0$, $y_0$, and orthonormal $V_0$ to fully parameterize $\hat{f}_{0}$ and $\widehat{\cX}_0$.
   \STATE \textbf{Output: $X_T, y_T$} 
   \FOR{$t = 0, 1, \ldots, T$}
   \STATE $(X_{t+1}, \nu_{t+1}) \gets {\argmax\limits_{(X, \nu) \in \widehat{\cX}_t \times \mathrm{N}} } \psi_t(X, \nu)$ 
   \STATE $ \tilde{y}_{t+1} \gets y_t - \frac{1}{\rho} (b - \nu_{t+1} - \cA X_{t+1})$
   \IF{$\beta (f(y_t) - \hat{f}_t(\tilde{y}_{t+1})) \leq f(y_t) - f(\tilde{y}_{t+1})$}
   \STATE $y_{t+1} \gets \tilde{y}_{t+1}$ \algocomment{descent step\qquad}
   \ELSE
   \STATE $y_{t+1} \gets y_t$ \algocomment{null step\qquad}
   \ENDIF
   \STATE Update $\hat{f}_{t+1}$ and $\widehat{\cX}_{t+1}$ using~\eqref{eq:subspace_def},~\eqref{eq:model_def},~\eqref{eq:x_bar_def},~\eqref{eq:v_next_def}.
   \ENDFOR
   \STATE \textbf{return} $X_T, y_T$
\end{algorithmic}
\end{algorithm}

\paragraph{Feasibility of $\tilde{y}_{t+1}$.} If we substitute $\nu_{t+1} = \proj_\mathrm{N}(b - \cA X_{t+1} - \rho y_t)$ into~\eqref{eq:dual_update}, it can be seen that the candidate $\tilde{y}_{t+1}$ is always feasible  
\begin{equation*}
    \big(\tilde{y}_{t+1}\big)_\cI = \max \left\{ \big(y_t\big)_\cI - \frac{1}{\rho}\Big(b_\cI - \cA_\cI X_{t+1}\Big), 0 \right\} \geq 0,
\end{equation*}
and is also complementary to the dual slack variable $\nu_{t+1}$, i.e. $\langle \tilde{y}_{t+1}, \nu_{t+1} \rangle = 0.$
This view allows us to see that we can equivalently express the candidate iterate as
\begin{equation*}
\tilde{y}_{t+1} = \proj_\cY\left( y_t - \frac{1}{\rho}(b - \cA X_{t+1})\right).
\end{equation*}
The complete algorithm is detailed in Algorithm~\autoref{alg:spectral_bundle}.
Notice that USBS has no iteration-dependent step-sizes or parameters, making it more amenable to effectively utilize a warm-start initialization.
It is important to note that the main cost of computing $f(y_t)$ and $f(\tilde{y}_{t+1})$ is a maximum eigenvalue computation of $C - \cA\adj y_t$ and $C - \cA\adj \tilde{y}_{t+1}$, respectively.
Lastly, we compute $\hat{f}_t(\tilde{y}_{t+1})$ using a primal-dual interior point method derived in Appendix~\ref{sec:ipm_lb_spec_est}.

\paragraph{Convergence Rate.} Under any setting of the parameters, Theorem~\ref{thm:sublinear_convergence} guarantees sublinear convergence for both primal and dual problems, showing that the iterates $X_t$ and $y_t$ converge in terms of objective gap and feasilibity.

\begin{theorem}
    \label{thm:sublinear_convergence}
    Suppose strong duality holds. For any fixed $ \rho > 0, \beta \in (0, 1), k_c \geq 1, k_p \geq 0$, and $\alpha \geq 2 \mathfrak{N}(\cX_\star)$, USBS produces iterates $X_t \succeq 0$ and $y_t \in \cY$ such that for any $\varepsilon \in (0, 1]$,
    \begin{align*}
    \textrm{penalized dual optimality: }&f(y_t) - f(y_\star) \leq \varepsilon, \\
    \textrm{primal feasiblity: }& \dist(\cA X_t, \cK) \leq \sqrt{\varepsilon}, \\
    \textrm{dual feasiblity: }& \lambda_{\max}(C - \cA\adj y_t) \leq \varepsilon, \\
    \textrm{primal-dual optimality: }& | \langle b, y_t \rangle - \langle C, X_t \rangle | \leq \sqrt{\varepsilon},
    \end{align*}
    by some iteration $\cO(1 / \varepsilon^3)$. And, if strict complementarity holds, then these conditions are achieved by some iteration $\cO(1 / \varepsilon)$. Additionally, if Slater's condition holds and $y_\star$ is unique, then the approximate primal feasibility and approximate primal-dual optimality respectively improve to 
    \begin{equation*}
    \dist(\cA X_t, \cK) \leq \varepsilon \quad \textrm{and} \quad | \langle b, y_t \rangle - \langle C, X_t \rangle | \leq \varepsilon.
    \end{equation*}
\end{theorem}
\noindent The proof is of this theorem is given in Appendix~\ref{sec:main_thm_proof}.
\section{Scaling with Matrix Sketching}
\label{sec:sketching}

The memory required to store the primal iterates $\bar{X}_t$ (similarly $X_t$) is prohibitive to scaling SDPs to large problem instances. 
We will utilize a \emph{Nystr\"om sketch}~\cite{tropp2017fixed, gittens2013topics, halko2011finding, li2017algorithm} to track a compressed low-rank projection of the primal iterate as it evolves which at any iteration can be used to compute a provably accurate low-rank approximation of $\bar{X}_t$.

First, notice that the operations carried out by USBS do not require explicitly storing $\bar{X}_t$.
In fact, if we are not interested in the primal iterates and we only want to solve the dual problem (e.g. in the case where we want to test the feasbility of the primal problem), we only need to store $\langle C, \bar{X}_t \rangle$, $\tr(\bar{X}_t)$, and $\cA \bar{X}_t$, which can be efficiently maintained with low-rank updates to $\bar{X}_t$ (see Appendix~\ref{sec:add_sketching} for more details).
This means if we are interested in the primal iterates, our storage of $\bar{X}_t$ is independent from the operations of USBS.

Consider any iterate $\bar{X}_t \in \SS^n_+$. Let $r \in \{1, \ldots, n\}$ be a parameter that trades off accuracy for scalability.
To construct the Nystr\"om sketch, we sample a random projection matrix $\Psi \in \RR^{n \times r}$ such that $\Psi_{ij} \sim \cN(0, 1)$ are sampled i.i.d.
The projection of $\bar{X}_t$, or \emph{sketch}, can be computed  as $P_t = \bar{X}_t \Psi \in \RR^{n \times r}$.
We can efficiently maintain this sketch $P_t$ under low-rank updates made to $\bar{X}_t$, 
\begin{equation}
    \begin{aligned}
    P_{t+1} \gets &\Big(\eta_{t+1} \bar{X}_t + V_t Q_{\underline{c}} \Lambda_{\underline{c}} Q_{\underline{c}}^\top V_t^\top\Big) \Psi \\
    &= \eta_{t+1} P_t + V_t Q_{\underline{c}} \Lambda_{\underline{c}} \Big(Q_{\underline{c}}^\top V_t^\top\Psi\Big).
    \end{aligned}
\end{equation}
\vspace*{-0.5cm}

Given the sketch $P_t$ and the projection matrix $\Psi$, we can compute a rank-$r$ approximation to $\bar{X}_t$.
It is important to note that by applying this Nystr\"om sketching technique the only approximation error is in the approximation of $\bar{X}_t$, and all of the error is constant (i.e. the error does not compound over time) and comes only from the projection matrix $\Psi$ (and the choice of $r$).
See Appendix~\ref{sec:add_sketching} for more details.
\section{Experiments}
\label{sec:experiments}
We evaluate USBS on three different problem types with and without warm-starting strategies against the scalable semidefinite programming algorithm CGAL~\cite{yurtsever2019conditional, yurtsever2021scalable} (with sketching).
~\citet{yurtsever2021scalable} perform an extensive evaluation against strong SDP solvers, including SeDuMi~\cite{sturm1999using}, SDPT3~\cite{toh1999sdpt3},  Mosek~\cite{aps2019mosek}, SDPNAL+\cite{yang2015sdpnal}, and the BM factorization heuristic~\cite{burer2003nonlinear}.
They show on several applications that CGAL with sketching is by far the standard against which to compare for scalable semidefinite programming.

\begin{table*}[h!]
\begin{center}
\caption{MaxCut data instance statistics.}
\scriptsize
\begin{tabular}{lcccccccccc}
\toprule
 & \textsf{fe\_sphere}  & \textsf{hi2010} & \textsf{fe\_body} & \textsf{me2010} & \textsf{fe\_tooth} & \textsf{598a} & \textsf{144} & \textsf{auto} & \textsf{netherlands\_osm} & \textsf{333SP}\\
\midrule
$n$ & 16K & 25K & 45K & 70K & 78K & 111K & 145K & 449K & 2.2M & 3.7M \\
$\textrm{nnz}(L)$ & 115K & 149K & 372K & 405K & 983K & 1.6M &  2.3M & 7.1M &  7.1M & 25.9M\\
\bottomrule
\end{tabular}
\label{tab:maxcut_dataset_stats}
\end{center}
\end{table*}
\begin{figure*}[h!]
\vspace{-0.5cm}
\includegraphics[width=\textwidth]{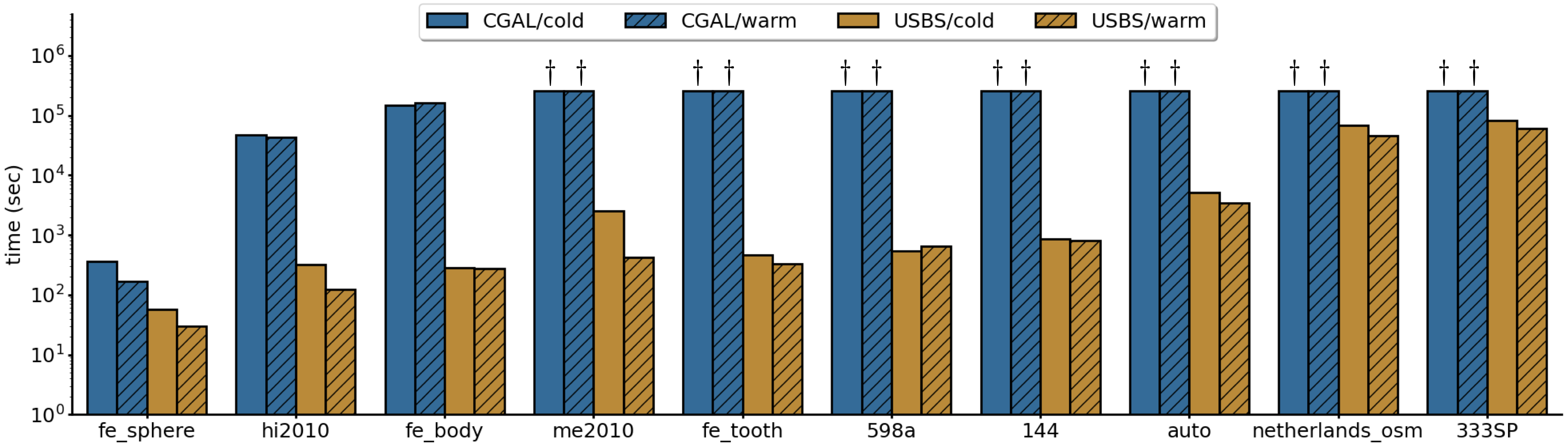}
\caption{\textbf{Convergence time (sec) to moderate relative error tolerance ($\downarrow$).} The time (in seconds) for CGAL and USBS to achieve an $\varepsilon$-approximate solution for $\varepsilon = 10^{-1}$ with and without warm-starting on 99\% of initial data on ten \texttt{DIMACS10} MaxCut instances. The bars marked with $\dagger$ indicate an $\varepsilon$-approximate solution was not achieved in 72 hours. The datasets are sorted in ascending order by $n$, ranging from 16K to 3.7M (more than $10^{13}$ decision variables for \smsf{333SP}).  Note that warm-starting generally improves convergence, but does not always (e.g. \smsf{598a}). We observe  USBS achieves an extraordinary improvement in convergence over CGAL which fails to reach an accurate solution on 7 out of 10 instances. In contrast, USBS is able to reach a solution on all of the problem instances in 28 hours or less without a warm-start initialization.}
\label{fig:maxcut_bar_chart}
\end{figure*}
We consider the primal iterate an $\varepsilon$-approximate solution if the relative primal objective suboptimality and relative infeasibility is less than $\varepsilon$, i.e.
\begin{equation}
    \label{eq:relative_error_tol}
    \frac{\langle C, X \rangle - \langle C, X_\star \rangle}{1 + \langle C, X \rangle} \leq \varepsilon \quad \textrm{and} \quad
    \frac{\dist(\cA X, \cK)}{1 + \|b\|} \leq \varepsilon.
\end{equation}
Computing the relative infeasibility is simple.
We do not usually have access to the optimal primal objective value $\langle C, X_\star \rangle$ to compute the relative primal objective suboptimality, but we can compute an upper bound. 
See~\eqref{eq:objective_subopt_upper_bound} for how we compute this upper bound.
We include further experimental details in Appendix~\ref{sec:add_expt_setup} and results in Appendix~\ref{sec:add_expt_results}.

\subsection{MaxCut}
The MaxCut problem is a fundamental combinatorial optimization problem and its SDP relaxation is a common test bed for SDP solvers.
Given an undirected graph, the MaxCut is a partitioning of the $n$ vertices into two sets such that the number of edges between the two subsets is maximized.
The MaxCut SDP relaxation~\cite{goemans1995improved} is as follows
%
%
\begin{equation}
    \label{eq:maxcut_sdp}
    \maximize \;\; \frac{1}{4} \, \tr(L X) \;\; \st \;\; \textrm{diag}(X) = \mathbf{1} \textrm{ and } X \succeq 0,
\end{equation}
where $L$ is the graph Laplacian, $\textrm{diag}(\cdot)$ extracts the diagonal of a matrix into a vector, and $\textbf{1} \in \RR^{n}$ is the all-ones vector.
In many instances of MaxCut, the graph Laplacian $L$ is sparse and the optimal solution is low-rank. 
This means that we can represent the MaxCut instance with far less than $\cO(n^2)$ memory and leverage sketching the primal matrix $X$ to improve scalability temendously.



We evaluate the CGAL and USBS with and without warm-starting on ten instances from the \texttt{DIMACS10}~\cite{dimacs10} where $n$ and the number of edges in each graph are shown in~\autoref{tab:maxcut_dataset_stats}.
We warm-start each method by dropping the last 1\% of vertices from the graph, resulting in a graph with 99\% of the vertices.
We pad the solution from the 99\%-sized problem with zeros and rescale as necessary to create the warm-start initialization.
\autoref{fig:maxcut_bar_chart} displays the time (in seconds) taken by CGAL and USBS, with and without warm-starting on each of the ten instances, where $r = 10$.



\subsection{Quadratic Assignment Problem}

The quadratic assignment problem (QAP) is a very difficult but fundamental class of combinatorial optimization problems containing the traveling salesman problem, max-clique, bandwidth problems, and facilities location problems among others~\cite{loiola2007survey}. 
SDP relaxations have been shown to facilitate finding good solutions to large QAPs~\cite{zhao1998semidefinite}. 
For any QAP instance, the goal is to optimize an assignment matrix $\Pi$ which aligns a weight matrix $W \in \SS^\texttt{n}$ and distance matrix $D \in \SS^\texttt{n}$. 
The number of $\smtt{n} \times \smtt{n}$ assignment matrices is $\smtt{n}!$, so a brute-force search becomes quickly intractable as \smtt{n} grows.
Generally, QAP instances with $\smtt{n} > 30$ are intractable to solve exactly.
 
There are many SDP relaxtions for QAPs, but we consider the one presented by~\cite{yurtsever2021scalable} and inspired by~\cite{huang2014scalable, bravo2018semidefinite} which is 
\begin{equation}
    \label{eq:qap_sdp}
    \begin{aligned}
        &\minimize  \;\, \tr \big( (D \otimes W) \, \smsf{Y} \big) \\
        &\;\st \;\;\, \tr_1(\smsf{Y}) = I, \; \tr_2(\smsf{Y}) = I, \; \mathcal{G}(\smsf{Y}) \geq 0, \\
        &\qquad\;\; \textrm{vec}(\smsf{B}) = \textrm{diag}(\smsf{Y}), \; \smsf{B} \textbf{1} = \textbf{1}, \; \textbf{1}^\top \smsf{B} = \textbf{1}^\top, \; \smsf{B} \geq 0, \\
        &\qquad\;\; X := \begin{bmatrix}
            1 & \textrm{vec}(\smsf{B})^\top \\
            \textrm{vec}(\smsf{B}) & \smsf{Y}
        \end{bmatrix} \succeq 0, \; \tr(\smsf{Y}) = \smtt{n}
    \end{aligned}
\end{equation}
where $\otimes$ denotes the Kronecker product, $\tr_1(\cdot)$ and $\tr_2(\cdot)$ denote the partial trace over the first and second systems of Kronecker product respectively, $\mathcal{G}(\smsf{Y})$ extracts the entries of $\smsf{Y}$ corresponding to the nonzero entries of $D \otimes W$, and $\textrm{vec}(\cdot)$ stacks the columns of a matrix one on top of the other to form a vector.
In many cases, one of $D$ or $W$ is sparse (i.e.~$\mathcal{O}(n)$ nonzero entries) resulting in $\mathcal{O}(\smtt{n}^3)$ total constraints for the SDP.

The primal variable $X$ has dimension $(\smtt{n}^2 + 1) \times (\smtt{n}^2 + 1)$ and as a result the SDP relaxation has $\mathcal{O}(\smtt{n}^4)$ decision variables.
Given the aggressive growth in complexity, most SDP based algorithms have difficulty operating on instances where $\smtt{n} > 50$.
To reduce the number of decision variables, we sketch $X$ with $r = \smtt{n}$, resulting in $\mathcal{O}(\smtt{n}^3)$ entries in the sketch and the same number of constraints in most natural instances.
We show that this enables CGAL and USBS to scale to instances where $\smtt{n} = 198$ (1.5 billion decision variables).

%
We evaluate CGAL and USBS on select large instances from \texttt{QAPLIB}~\cite{burkard1997qaplib} and \texttt{TSPLIB}~\cite{reinelt1995tsplib95} ranging in size from 136 to 198.
Provided with each of these QAP instances is the known optimum.
For many instances, we find that the quality of the permutation matrix produced by the rounding procedure does not entirely correlate with the quality of the iterates with respect to the SDP~\eqref{eq:qap_sdp}.
Thus, we apply the rounding procedure at every iteration of both CGAL and USBS.
Our metric for evaluation is \smtt{relative gap} which is computed as follows
\begin{equation*}
    \smtt{relative gap} = \frac{\smtt{upper bound obtained} - \smtt{optimum}}{\smtt{optimum}}.
\end{equation*}
We report the lowest value so far of \smtt{relative gap} as \smtt{best relative gap}.
CGAL~\cite{yurtsever2021scalable} is shown to obtain significantly smaller \smtt{best relative gap} than CSDP~\cite{bravo2018semidefinite} and PATH~\cite{zaslavskiy2008path} on most instances, and thus, is a strong baseline for achieving high quality approximate solutions.

To create a warm-start initialization, we create a slightly smaller QAP by dropping the final row and column of both $D$ and $W$ (i.e. solve a size $\smtt{n} - 1$ subproblem of the original instance).
We use the solution to slightly smaller problem to set a warm-start initialization for the original problem, rescaling and padding with zeros where necessary.
We stopped optimizing after one hour. When warm-starting, we optimize the slightly smaller problem for one hour and then optimize the original problem for one hour.
\autoref{fig:qap_pr144_relative_gap} plots the \smtt{relative gap} and \smtt{best relative gap} against time for CGAL and USBS both with and without warm-starting for one \texttt{TSPLIB} instance.


\begin{figure}[t!]
\subfloat{%
  \includegraphics[width=0.48\textwidth]{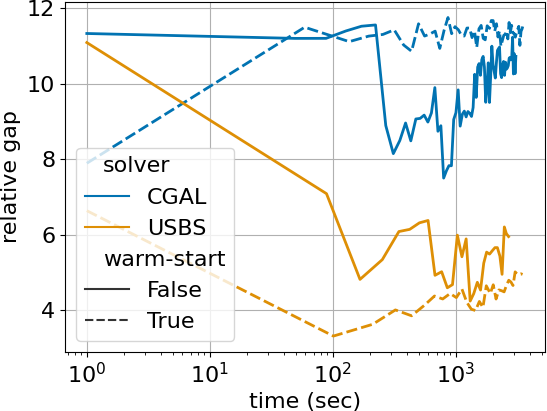}
}
\hfill
\subfloat{%
  \includegraphics[width=0.48\textwidth]{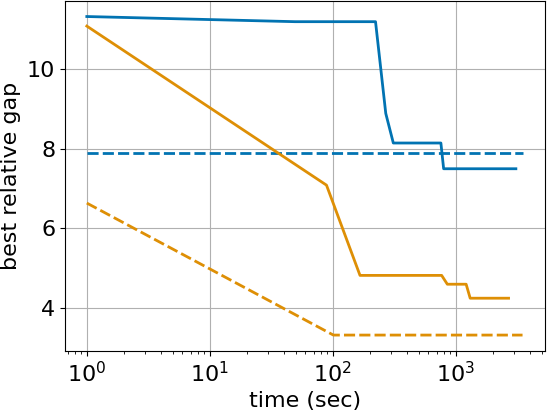}
}
\caption{\textbf{\smtt{relative gap} ($\downarrow$) vs.~time.} We plot the \smtt{relative gap} (y-axis, left) and \smtt{best relative gap} (y-axis, right) against time in seconds (x-axis) for one QAP instance, \smsf{pr144}, from \texttt{TSPLIB} ($\smtt{n} = 144$) over one hour of optimization. We observe that for both algorithms the best rounded solution is found early in optimization. We observe that USBS is able to leverage a warm-start initialization.}
\label{fig:qap_pr144_relative_gap}
\vskip -0.1in
\end{figure}

%

\subsection{Interactive Entity Resolution with $\exists$-constraints}
\citet{pmlr-v162-angell22a} introduced a novel SDP relaxation of a combinatorial optimization problem that arises in automatic knowledge base construction ~\citep{uhlen2010towards,krishnamurthy2015learning,iv2022fruit}. 
They consider the problem of interactive entity resolution with user feedback to correct predictions made by a machine learning algorithm.
Specifically, they introduce \emph{\econstraints\!\!}, a new paradigm of interactive feedback for correcting entity resolution predictions and presented a novel SDP relaxation as part of a heuristic algorithm for satisfying \econstraints in the predicted entity resolution decisions.
This novel form of feedback allows users to specify constraints stating the existence of an entity with and without certain features.
As each \econstraint is added to the optimization problem additional constraints are added to the SDP relaxation.
We use the solution to the SDP without the newest \econstraint to warm-start the solver with the new \econstraint added.
See Appendix~\ref{sec:add_ecc} for more details.

We evaluate the performance of the CGAL and USBS with and without warm-starting on the same three author coreference datasets used in~\cite{pmlr-v162-angell22a}.
In these datasets, each mention is an author-paper pair and the goal is to identify which author-paper pairs refer to the same author.
%
We simulate \econstraint generation using the same oracle implemented in~\cite{pmlr-v162-angell22a}.
\autoref{fig:ecc_cumulative_solve_time} shows the cumulative SDP solve time to reach an $\varepsilon$-approximate solution against the number of \econstraints for each of the three datasets until the perfect entity resolution decsions were predicted by the inference algorithm.

\begin{table}[t]
\begin{center}
\caption{Author coreference dataset statistics.}
\small
\begin{tabular}{lccccc}
\toprule
 & \# mentions & \# blocks & \# clusters & \textrm{nnz}($W$) & \# features \\
\midrule
\bf PubMed & 315 & 5 & 34 & 3,973 & 14,093\\
\bf QIAN & 410 & 38 & 77 & 5,158 & 10,366\\
 \bf SCAD-zbMATH\qquad & 1,196 & 120 & 166 & 18,608 & 8,203\\
\bottomrule
\end{tabular}
\label{tab:ecc_dataset_stats}
\end{center}
\end{table}
\begin{figure*}[t!]
\subfloat[PubMed\label{subfig:pubmed}]{%
  \includegraphics[width=0.32\textwidth]{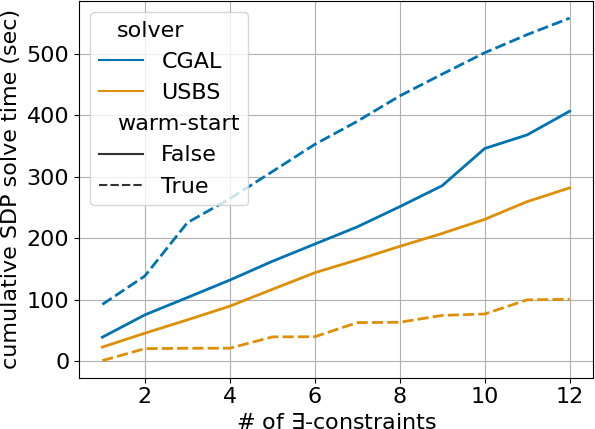}
}
\hfill
\subfloat[QIAN\label{subfig:rand_idx_qian}]{%
  \includegraphics[width=0.32\textwidth]{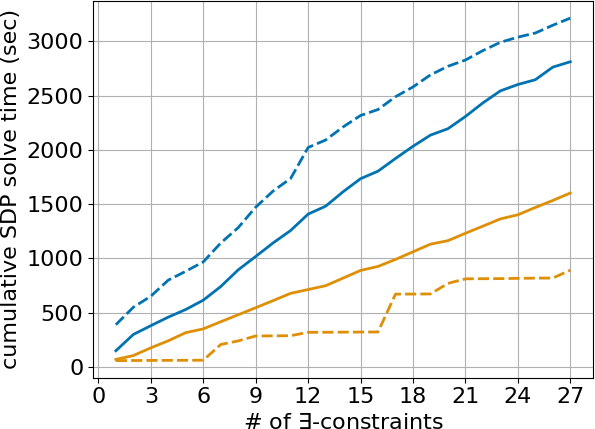}
}
\hfill
\subfloat[SCAD-zbMATH\label{subfig:rand_idx_zbmath}]{%
  \includegraphics[width=0.32\textwidth]{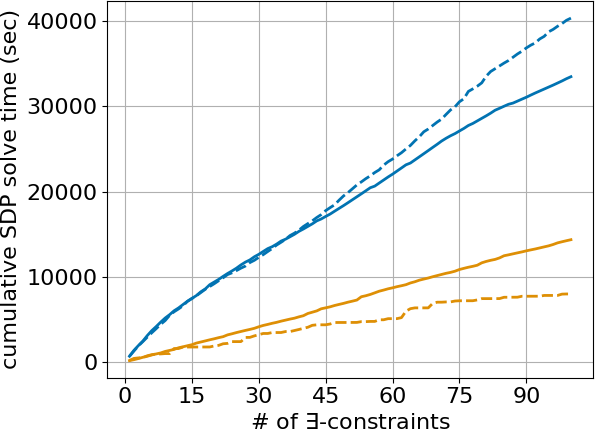}
}
\caption{\textbf{Cumulative SDP solve time ($\downarrow$) vs.~number of \econstraints\!\!}. In each plot (one for each author coreference dataset), the $x$-axis is the number of \econstraints generated one after the other over time and the $y$-axis is the the cumulative solve time (in seconds) for each SDP solver to reach a relative suboptimality, relative infeasibility, and max absolute infeasibility (i.e. $\|\cA X - \proj_\cK(\cA X)\|_\infty \leq \varepsilon$) of $\varepsilon = 10^{-1}$. When warm-starting, both solvers are initialized using the solution from the previous SDP (with one less \econstraint\!\!). \econstraints are generated until the perfect clustering is predicted. We observe that USBS is able to leverage a warm-start initialization. In addition, we observe that the performance gap between USBS and CGAL grows as the problem size grows. See \autoref{tab:ecc_dataset_stats} for dataset sizes and details.}
\label{fig:ecc_cumulative_solve_time}
\end{figure*}
\section{Related Work}

\paragraph{Efficient SDP solving.}
Due to their wide-spread applicability and importance, methods for solving SDPs efficiently have been extensively studied in the literature \cite{todd2001semidefinite,nesterov1989self,nesterov1994interior,alizadeh1995interior, boyd2011distributed,friedlander2016low, yurtsever2019conditional, yurtsever2015universal,ding2021simplicity}.
Interior point methods~\cite{helmberg1994interior,nesterov1994interior,alizadeh1995interior} are by far the most widely-used and well-known method for solving SDPs.
Interior point methods enjoy quadratic convergence, but have incredibly high per-iteration complexity, and thus, are impractical for instances having more than a few hundred variables.
ADMM~\cite{boyd2011distributed} and splitting-based methods~\cite{scs} have also been applied to solving SDPs.
These methods are generally more applicable than interior point methods, but our still restricted do to the computationally expensive full eigenvalue decomposition.
First-order methods such as conditional gradient augmented Lagrangian~\cite{yurtsever2017sketchy}, primal-dual subgradient algorithms~\cite{yurtsever2015universal}, the matrix multiplicative weight method~\cite{arora2005fast,tsuda2005matrix}, and the mirror-prox algorithm with the quantum entropy mirror map~\cite{nemirovski2004prox} have by far the lowest per-iteration complexity among all of the methods for solving SDPs.
This lower per-iteration complexity leads to slower convergence rate for these first-order methods.

\textbf{Spectral bundle methods.}
\citet{helmberg2000spectral} were the first to introduce spectral bundle methods for equality-constrained SDPs. The biggest difference between the original spectral bundle method and our approach is the model where $k_c = 1$ is fixed and $k_p$ is chosen by the user. We find that larger $k_c$ provides better convergence in general and $k_p$ is not very helpful (and sometimes harmful) to convergence. Several other algorithmic variants of spectral bundle methods exist including solving inequality-constrained SDPs~\cite{helmberg2002spectral}, trust region-based methods for nonconvex eigenvalue optimization~\cite{apkarian2008trust}, and incorporating second-order information to speed up empirical convergence~\cite{helmberg2014spectral}.
\citet{ding2023revisiting} present a spectral bundle method most similar to ours for solving equality-constrained SDPs. They prove an impressive set of convergence rates under various assumptions, including linear convergence if certain strong assumptions hold. 
The biggest difference between previous work of spectral bundle methods and USBS is the fact that USBS has more flexibility and scalability as compared with other methods both in terms of the constraints USBS supports in the SDP and the model the user can specify.
In addition, this work addresses many of the practical implementation questions enabling scalability to instances with multiple orders of magnitude more variables.

\textbf{Memory-efficient semidefinte programming.}
The most widely-known memory efficient approach to semidefinite programming is the Burer–Monteiro (BM) factorization heuristic~\cite{burer2003nonlinear}.
Several optimization techniques have been used to optimize the low-rank factorized problem~\cite{wang2017mixing,boumal2014manopt,burer2005local,kulis2007fast,sahin2019inexact,souto2022exploiting}.
\citet{ding2021optimal} present an optimal storage approach to solving SDPs by first approximately solving the dual problem and then using the dual slack matrix to solve the primal problem efficiently.
\citet{yurtsever2021scalable} implement the conditional gradient augmented Lagrangian method~\cite{yurtsever2019conditional} with the matrix sketching technique in~\cite{yurtsever2017sketchy} which forms a strong state-of-the-art method for scalably solving general SDPs.
\citet{ding2023revisiting} implements a spectral bundle method for solving equality constrained SDPs with the same matrix sketching technique~\cite{yurtsever2017sketchy}.
\section{Conclusion}

In this work, we presented a practical spectral bundle method for solving large SDPs with both equality and inequality constraints.
We proved non-asymptotic convergence rates under standard assumptions for USBS.
We showed empirically that USBS is fast and scalable on  practical SDP instances.
Additionally, we showed that USBS can more reliably leverage a warm-start initialization to accelerate convergence. Lastly, we make our standalone implementation in pure JAX available for wide-spread use.


\section*{Acknowledgements}
This work is funded in part by the Center for Data Science and the Center for Intelligent Information Retrieval at the University of Massachusetts Amherst, and in part by the National Science Foundation under grants IIS-1763618 and IIS-1922090, and in part by the Chan Zuckerberg Initiative under the project Scientific Knowledge Base Construction, and in part by IBM
Research AI through the AI Horizons Network.
The work reported here was performed in part by the Center for Data Science and the Center for Intelligent Information Retrieval, and in part using high performance computing equipment obtained under a grant from the Collaborative R\&D Fund managed by the Massachusetts Technology Collaborative.
Rico Angell is partially supported by the National Science Foundation Graduate Research Fellowship under grant 1938059.
Any opinions, findings and conclusions or recommendations expressed in this material are those of the authors and do not necessarily reflect those of the sponsor(s).

\bibliographystyle{plainnat}
\bibliography{references}

\pagebreak

\appendix

\section{Derivations}
\label{sec:derivations}

\subsection{Penalized Dual Objective}
\label{sec:derive_pen_dual}

Begin by considering the following Lagrangian formulation of the model problem
\begin{equation}
    p_\star = \max_{X \in \SS^n_+} \; \min_{y \in \RR^m : y_{\cI} \geq 0} \; \cL(X, y) =  \min_{y \in \RR^m : y_{\cI} \geq 0} \; \max_{X \in \SS^n_+} \; \cL(X, y) = d_\star,
\end{equation}
where the Lagrangian is defined as
\begin{equation}
    \cL(X, y) := \langle C - \cA\adj  y, X \rangle + \langle b, y \rangle.
\end{equation}
Since $\cX_\star \subset \cX$, we know
\begin{equation}
\label{eq:restricted_domain}
\max_{X \in \SS^n_+} \; \min_{y \in \cY} \; \cL(X, y) = 
\max_{X \in \cX} \; \min_{y \in \cY} \; \cL(X, y)
  =\min_{y \in \cY}  \; \max_{X \in \cX} \; \cL(X, y).
\end{equation}
Furthermore, it is easy to show the following fact~\cite{overton1992large},
\begin{equation}
    \label{eq:max_eigenvalue_equiv}
    \max_{X \succeq 0 \, , \, \tr(X) \leq 1} \langle Z, X \rangle = \max \{ \lambda_\textrm{max}(Z), 0\}.
\end{equation}
Incorporating~\eqref{eq:max_eigenvalue_equiv} with~\eqref{eq:restricted_domain} allows us to write an equivalent formulation of the original dual problem by defining the penalized dual objective as shown in~\eqref{eq:pen_dual}
\begin{equation}
    f(y) := \alpha \max\{\lambda_\textrm{max}(C - \cA\adj y), 0\} + \langle b, y \rangle + \iota_\cY(y),\tag{pen-D}
\end{equation}
where $\iota_\cY(\cdot)$ is the indicator function defined such that $\iota_\cY(y) = 0$ if $y \in \cY$ and $\iota_\cY(y) = +\infty$ otherwise.

\subsection{Equivalent Penalized Dual Objective}
\label{sec:equiv_derive_pen_dual}
We begin by considering the dual form of the indicator function $\iota_\cY(\cdot)$ 
\begin{equation}
    \label{eq:dual_indicator}
    \iota_\cY(y) = \sup \{ - \langle \nu, y \rangle : \nu_\cI \leq 0, \nu_{\cI'} = 0\}.
\end{equation}
Define $\mathrm{N} := \{\nu \in \RR^m : \nu_{\cI} \leq 0, \nu_{\cI^{'}} = 0\}$. Substituting~\eqref{eq:dual_indicator} into~\eqref{eq:pen_dual} and utilizing~\eqref{eq:max_eigenvalue_equiv} gives
\begin{equation*}
    f(y) = \sup_{(X, \nu) \in \cX \times \mathrm{N}} \langle C - \cA\adj y , X\rangle + \langle b - \nu, y \rangle.
\end{equation*}

\subsection{Candidate Iterate}
\label{sec:derive_cand_iterate}

Notice that the candidate iterate can be written as follows
\begin{equation}
\tilde{y}_{t+1} = \argmin_{y \in \RR^m} \sup_{(X, \nu) \in \widehat{\cX}_t \times \mathrm{N}} F(X, \nu, y),
\end{equation}
where
\begin{equation*}
F(X, \nu, y) := \langle C - \cA\adj  y, X \rangle + \langle b - \nu, y \rangle + \frac{\rho}{2} \| y - y_t \|^2.
\end{equation*}

To solve for $\tilde{y}_{t+1}$, start by fixing $(X, \nu) \in \widehat{\cX}_t \times \mathrm{N}$ arbitrarily. Then, completing the square gives
\begin{equation*}
   \begin{aligned}
       \argmin_{y \in \RR^m} F(X, \nu, y)
       &= \argmin_{y \in \RR^m} \langle b - \nu - \cA X, y \rangle + \frac{\rho}{2} \| y - y_t \|^2 \\
       &= \argmin_{y \in \RR^m} \frac{\rho}{2} \left\| y - y_t + \frac{1}{\rho} (b - \nu - \cA X)\right\|^2 \\
       &= y_t - \frac{1}{\rho} (b - \nu - \cA X).
   \end{aligned} 
\end{equation*}
This implies that the candidate iterate can be computed as follows
\begin{equation}
    \tilde{y}_{t+1} = \argmin_{y \in \RR^m} \hat{f}_t(y) + \| y - y_t\|^2\\
    = y_t - \frac{1}{\rho} (b - \nu_{t+1} - \cA X_{t+1}),
\end{equation}
where $(X_{t+1}, \nu_{t+1}) \in \widehat{\cX}_t \times \mathrm{N}$ is the solution to the following optimization problem
\begin{equation}
(X_{t+1}, \nu_{t+1}) \in \argmax_{(X, \nu) \in \widehat{\cX}_t \times \mathrm{N}} \psi_t(X, \nu),
\end{equation}
where
\begin{equation}
    \psi_t(X, \nu) := 
\langle C, X \rangle + \langle b - \nu - \cA X , y_t\rangle - \frac{1}{2\rho} \| b - \nu - \cA X \|^2.
\end{equation}

\section{Solving Iteration Subproblems}
\label{sec:ipm}
In this section, we detail the primal-dual path-following interior point methods used to solve the small semidefinite program to compute the value $\hat{f}_t(\tilde{y}_{t+1})$ and the small quadratic semidefinite program~\eqref{eq:small_quad_sdp}.
Before we can derive the algorithms, it is necessary to define the $\svec$ operator and symmetric Kronecker product~\cite{alizadeh1998primal,schacke2004kronecker,todd1998nesterov}.

For any matrix $A \in \SS^n$, the vector $\svec(A) \in \RR \mathclose{\vphantom{\big)}}^{n+1 \choose 2}$ is defined as
\begin{equation*}
\svec(A) = \left[ a_{11}, \sqrt{2} a_{21}, \ldots, \sqrt{2} a_{n1}, a_{22}, \sqrt{2} a_{32}, \ldots, \sqrt{2} a_{n2}, \ldots, a_{nn}\right]^\top.
\end{equation*}
The $\svec$-operator is a structure preserving map between $\SS^n$ and $\RR\mathclose{\vphantom{\big)}}^{n+1 \choose 2}$ where the constant $\sqrt{2}$ multiplied by some of the entries ensures that
\begin{equation*}
\langle A, B \rangle = \tr(AB) = \svec(A)^\top \svec(B), \quad \forall\, A, B \in \SS^n.
\end{equation*}
For any $M \in \RR \mathclose{\vphantom{)}}^{n \times n}$, let $\textrm{vec}(M)$ be the map from $\RR \mathclose{\vphantom{)}}^{n \times n}$ to $\RR \mathclose{\vphantom{)}}^{n^2}$ defined by stacking the columns of $M$ into a single $n^2$-dimensional vector.
It is also useful to define the matrix $ U \in \RR\mathclose{\vphantom{\big)}}^{{n+1 \choose 2} \times n^2}$ which maps $\textrm{vec}(A) \mapsto \svec(A)$ for any $A \in \SS^n$. Let $u_{ij,kl}$ be the entry in the row which defines element $a_{ij}$ in $\svec(A)$ and the column that is multiplied with the element $a_{kl}$ in $\textrm{vec}(A)$.
Then
\begin{equation*}
    u_{ij,kl} = \begin{cases}
        1 & \quad i = j = k = l \\
        \frac{1}{\sqrt{2}} & \quad i = k\ne j = l \textrm{, or } i = l \ne j = k \\
        0 & \quad \textrm{o.w.}
    \end{cases}
\end{equation*}
As an example, the case when $n=2$:
\begin{equation*}
  U = \begin{bmatrix}
      1 & 0& 0& 0 \\
      0 & \frac{1}{\sqrt{2}}& \frac{1}{\sqrt{2}} & 0 \\
      0 & 0& 0& 1 \\
  \end{bmatrix}.  
\end{equation*}
Note that $U$ is a unique matrix with orthonormal rows and has the following property
\begin{equation*}
    U^\top U \textrm{vec}(A) = U^\top \svec(A) = \textrm{vec}(A), \quad \forall \, A \in \SS^n.
\end{equation*}

The symmetric Kronecker product $\otimes_s$ can be defined for any two square matrices $G, H \in \RR^{n\times n}$ by its action on a vector $\svec(A)$ for $A \in \SS^n$ as follows
\begin{equation*}
(G \otimes_s H)\, \svec(A) = \frac{1}{2} \svec(HAG^\top + GAH^\top).
\end{equation*}
Alternatively, but equivalently~\cite{schacke2004kronecker}, the symmetric Kronecker product can be defined more explicitly using the matrix $U$ defined above as follows
\begin{equation*}
G \otimes_s H = \frac{1}{2} U( G \otimes H + H \otimes G) U^\top,
\end{equation*}
where $\otimes$ is the standard Kronecker product. We use this latter definition in our implementation.

\subsection{Computing $\hat{f}_t(\tilde{y}_{t+1})$}
\label{sec:ipm_lb_spec_est}
The value $\hat{f}_t(\tilde{y}_{t+1})$ is the optimum value of the following optimization problem (remember that $\nu$ can be dropped since $\tilde{y}_{t+1}$ is always feasible)
\begin{equation*}
\begin{aligned}
&\maximize \;\; \langle C - \cA\adj \tilde{y}_{t+1}, \eta \bar{X}_t + V_t S V_t^\top \rangle + \langle b, \tilde{y}_{t+1}\rangle\\
&\;\st \;\;\; \eta \geq 0 \\
&\qquad\;\; S \succeq 0 \\
&\qquad\;\; \eta + \tr(S) \leq \alpha
\end{aligned}
\end{equation*}
This amounts to computing $\langle C - \cA\adj \tilde{y}_{t+1}, \eta_\star \bar{X}_t + V_t S_\star V_t^\top \rangle + \langle b, \tilde{y}_{t+1}\rangle$ where $(\eta_\star, S_\star)$ is a solution to the following (small) semidefinite program
\begin{equation}
\label{eq:compressed_subprob_1}
\begin{aligned}
&\minimize \;\; g_1^\top \svec(S) + \eta \, g_2\\
&\;\st \;\;\; \eta \geq 0 \\
&\qquad\;\; S \succeq 0 \\
&\qquad\;\; 1 - v_I^\top \svec(S) - \eta \geq 0
\end{aligned}
\end{equation}
where 
\begin{equation*}
    \begin{aligned}
        g_1 &= \alpha \,\svec(V_t^\top\!(\cA\adj\tilde{y}_{t+1} - C) V_t), \\
        g_2 &= \frac{\alpha}{\tr\!\left(\bar{X}\right)} \langle \bar{X}_t, \cA\adj\tilde{y}_{t+1} - C) \rangle, \\
        v_I &= \svec(I).  \qquad (I \textrm{ is the } k \times k \textrm{ identity matrix})
    \end{aligned}
\end{equation*}

We will use a primal-dual interior point method to solve~\eqref{eq:compressed_subprob_1}.
We follow the well known technique for deriving primal-dual interior point methods~\cite{helmberg1994interior}. 
We start by defining the Lagrangian of the dual barrier problem of~\eqref{eq:compressed_subprob_1},
\begin{equation}
\label{eq:dual_barrier_Lagrangian1}
\begin{aligned}
L_\mu(S, \eta, T, \zeta, \omega) &= g_1^\top \svec(S) + \eta g_2 - \svec(S)^\top\svec(T) - \eta \zeta \\
&\qquad- \omega(1 - v_I^\top \svec(S) - \eta) + \mu (\log \det (T) + \log\zeta + \log \omega),
\end{aligned}
\end{equation}
where we introduce a dual slack matrix $T \succeq 0$ as complementary to $S$, a dual slack scalar $\zeta \geq 0$ as complementary to $\eta$, a Lagrange multiplier $\omega \geq 0$ for the trace constraint inequality, and a barrier parameter $\mu > 0$.
Notice that we have moved from needing to optimize a constrained optimization problem to an unconstrained optimization problem.
The saddle point solution of~\eqref{eq:dual_barrier_Lagrangian1} is given by the solution of the KKT-conditions, which reduces to just the first-order optimality conditions since the problem is unconstrained.
The first-order optimality conditions of~\eqref{eq:dual_barrier_Lagrangian1} are the following system of equations
\begin{align}
\nabla_S L_\mu &= g_1 - \svec(T) + \omega v_I = 0 \label{eq:kkt_1}\\
\nabla_\eta L_\mu &= g_2 - \zeta + \omega = 0 \label{eq:kkt_2}\\
\nabla_T L_\mu &= S - \mu T^{-1} = 0 \label{eq:kkt_3}\\
\nabla_\zeta L_\mu &= \eta - \mu \zeta^{-1} = 0 \label{eq:kkt_4}\\
\nabla_\omega L_\mu &= 1 - v_I^\top \svec(S) - \eta - \mu \omega^{-1} = 0 \label{eq:kkt_5}
\end{align}
By the strict concavity of $\log \det T$, $\log \zeta$, and $\log \omega$, there exists a unique solution $(S_\mu, \eta_\mu, T_\mu, \zeta_\mu, \omega_\mu)$ to this system of equations for any value of the barrier parameter $\mu > 0$.
The sequence of these solutions as $\mu \to 0$ forms the \emph{central trajectory} (also known as the central \emph{path}).
For a point $(S, \eta, T, \zeta, \omega)$ on the central trajectory, we can use any combination of~\eqref{eq:kkt_3},~\eqref{eq:kkt_4}, and/or~\eqref{eq:kkt_5} to solve for $\mu$,
\begin{equation}
    \label{eq:barrier_param_est}
    \mu = \frac{\langle S, T \rangle}{k} = \eta \zeta = \omega(1 - v_I^\top \svec(S) - \eta) = \frac{\langle S, T \rangle + \eta \zeta + \omega(1 - v_I^\top \svec(S) - \eta)}{k + 2}.
\end{equation}

The fundamental idea of primal-dual interior point methods is to use Newton's method to follow the central path to a solution of~\eqref{eq:compressed_subprob_1}.
Before we can apply Newton's method, we must linearize the non-linear equations ~\eqref{eq:kkt_3},~\eqref{eq:kkt_4}, and~\eqref{eq:kkt_5} into an equivalent linear formulation.
There are several ways one could linearize these equations and the choice of linearization significantly impacts the algorithm's behavior (see~\cite{alizadeh1998primal} or~\cite{helmberg1994interior} for more on the choice of linearization).
We choose one standard method linearization, detailed in the following system
of equations
\begin{equation}
    \label{eq:linearized_system1}
    \mathrm{F}_\mu(\theta) = \mathrm{F}_\mu(S, \eta, T, \zeta, \omega) := \begin{pmatrix}
        \vspace{0.1em} g_1 - \svec(T) + \omega v_I \\
        \vspace{0.1em} g_2 - \zeta + \omega \\
        \vspace{0.1em} ST - \mu I \\
        \vspace{0.1em} \eta\zeta - \mu\\
        \vspace{0.1em} \omega (1 - v_I^\top \svec(S) - \eta) - \mu\\
    \end{pmatrix}
    = 0.
\end{equation}
The solution $\theta_\star$ to this system of equations $\mathrm{F}_\mu(\theta) = 0$ satisfies the first-order optimality conditions~\eqref{eq:kkt_1}-\eqref{eq:kkt_5} and is the optimal solution to the barrier problem.
We utilize Newton's method to take steps in the update direction $\Delta \theta = (\Delta S, \Delta \eta, \Delta T, \Delta \zeta, \Delta \omega)$ towards $\theta_\star$.
The update direction $\Delta \theta$ determined by Newton's method must satisfy the following equation
\begin{equation*}
    \mathrm{F}_\mu(\theta) + \nabla \mathrm{F}_\mu(\Delta \theta) = 0.
\end{equation*}
Hence, the update direction $\Delta \theta$ is the solution to the following system of equations (after the same standard linearization has been applied to the following system as in~\eqref{eq:linearized_system1})
\begin{align}
- \svec(\Delta T) + \Delta \omega \, v_I &= \svec(T) - g_1 - \omega \, v_I \\
- \Delta \zeta + \Delta \omega &= \zeta - g_2 - \omega \\
\omega^{-1}(1 - v_I^\top \svec(S) - \eta) \Delta \omega - v_I^\top \svec(\Delta S) - \Delta \eta &= \mu \omega^{-1} + v_I^\top \svec(S) + \eta - 1 \\
(T \otimes_s S^{-1}) \, \svec(\Delta S) + \svec(\Delta T) &= \mu \,\svec(S^{-1}) - \svec(T) \\
\zeta \eta^{-1} \Delta \eta + \Delta \zeta &= \mu \eta^{-1} - \zeta
\end{align}
We can solve this system efficiently by analytically eliminating all variables except $\svec(\Delta S)$. We can compute $\svec(\Delta S)$ by solving the following linear matrix equation using an off-the-shelf linear system solver
\begin{equation}
\label{eq:delta_S_linear_system}
\begin{aligned}
&\left( T \otimes_s S^{-1} + \frac{\zeta \eta^{-1}}{\kappa_1 \zeta \eta^{-1} + 1} v_I v_I^\top \right) \svec(\Delta S) 
= v_I g_2 - v_I \mu \eta^{-1} - g_1 + \mu \, \svec(S^{-1}) \\
&\qquad\qquad + v_I \zeta \eta^{-1} (- \kappa_1 \zeta \eta^{-1} - 1)^{-1} \left(-\kappa_1 (\mu \eta^{-1} - g_2 - \omega) + \mu \omega^{-1} + v_I^\top \svec(S) + \eta - 1\right)
\end{aligned}
\end{equation}
where $\kappa_1 := \omega^{-1}(1 - v_I^\top \svec(S) - \eta)$.
Then, computing the rest of the update directions $\Delta \eta$, $\svec(\Delta T)$, $\Delta \zeta$, and $\Delta \omega$ amounts to back-substituting the solution to~\eqref{eq:delta_S_linear_system} for $\svec(\Delta S)$ through the analytical variable elimination equations.

Given how to compute the update directions, the primal-dual interior point method proceeds as follows. Initialize the variables $\theta = (S, \eta, T, \zeta, \omega)$ to an arbitrary strictly feasible point (i.e. $S \succ 0$, $\eta > 0$, $T \succ 0$, $\zeta > 0$, and $\omega > 0$). Starting from this primal-dual pair we compute an estimate of the barrier parameter as follows
\begin{equation*}
    \mu \gets \frac{\langle S, T \rangle + \eta \zeta + \omega (1 - v_I^\top \svec(S) - \eta)}{2(k + 2)},
\end{equation*}
where, as done in~\cite{helmberg1994interior}, we use~\eqref{eq:barrier_param_est} and divide by two.
Then, we compute the update direction $\Delta \theta$ as described above and perform a backtracking line search to find a step size $\delta \in (0, 1]$ such that $\theta + \delta \, \Delta \theta$ is again strictly feasible.
Lastly, following~\cite{helmberg2000spectral}, we compute a non-increasing estimate of the barrier parameter 
\begin{equation*}
    \mu \gets \min \left\{ \mu_\textrm{prev}, \gamma \frac{\langle S, T \rangle + \eta \zeta + \omega (1 - v_I^\top \svec(S) - \eta)}{2(k + 2)}\right \}
    \quad \textrm{where} \quad 
    \gamma = \begin{cases}
    1 &  \textrm{if }\delta \leq \frac{1}{5} \\
    \frac{5}{10} - \frac{4}{10} \delta^2 & \textrm{if } \delta > \frac{1}{5}
    \end{cases}.
\end{equation*}
We iterate over these steps until the barrier parameter $\mu$ is small enough (e.g. $\mu < 10^{-7}$).

\subsection{Solving~\eqref{eq:small_quad_sdp}}
\label{sec:small_quad_sdp_ipm}
The subproblem~\eqref{eq:small_quad_sdp} can be rewritten as the following (small) quadratic semidefinite program
\begin{equation*}
\begin{aligned}
&\maximize \;\;\langle C - \cA\adj y_t, \eta \bar{X}_t + V_t S V_t^\top \rangle + \langle b - \tilde{\nu}, y_t \rangle - \frac{1}{2\rho} \left\| b - \tilde{\nu} - \cA (\eta \bar{X}_t + V_t S V_t^\top)\right\|_2^2\\
&\; \st \;\;\; \eta \geq 0 \\
&\qquad\;\; S \succeq 0 \\
&\qquad\;\; \eta + \tr(S) \leq \alpha
\end{aligned}
\end{equation*}
For this subproblem, unlike the subproblem solved in~\autoref{sec:ipm_lb_spec_est}, we are solving for $\eta$ and $S$ to compute the candidate iterate $\tilde{y}_{t+1}$, update the primal variable, and  update the model.
This (small) quadratic semidefinite program is equivalent to the following optimization problem

\begin{equation}
\label{eq:compressed_subprob_2}
\begin{aligned}
&\minimize \;\; \frac{1}{2} \,\svec(S)^\top Q_{11} \,\svec(S) + \eta \, q_{12}^\top \, \svec(S) + \frac{1}{2} \eta^2 q_{22} + h_1^\top \svec(S) + \eta \, h_2\\
&\;\st \;\;\; \eta \geq 0 \\
&\qquad\;\; S \succeq 0 \\
&\qquad\;\; 1 - v_I^\top \svec(S) - \eta \geq 0
\end{aligned}
\end{equation}
where 
\begin{equation*}
    \begin{aligned}
        Q_{11} &= \frac{\alpha^2}{\rho} \sum_{i=1}^m \svec(V_t^\top\! A_i V_t) \,\svec(V_t^\top\! A_i V_t)^\top \\
        q_{12} &= \frac{\alpha^2}{\rho \, \tr\!\left(\bar{X}_t\right)} \,\svec\left(V_t^\top\! \cA^{*}\cA\,\bar{X}_t V_t\right) \\
        q_{22} &= \frac{\alpha^2}{\rho \, \tr\!\left(\bar{X}_t\right)^2} \left\langle \cA\bar{X}_t,\cA\bar{X}_t\right\rangle \\
        h_1 &= \alpha \,\, \svec\left(V_t^\top \!\left(\cA\adj y_t - C - \frac{1}{\rho} \cA\adj (b - \tilde{\nu}) \right) V_t\right) \\
        h_2 &= \frac{\alpha}{\tr\!\left(\bar{X}_t\right)} \left\langle \bar{X}, \, \cA\adj y_t - C - \frac{1}{\rho} \cA\adj (b - \tilde{\nu}) \right\rangle \\
        v_I &= \svec(I) 
    \end{aligned}
\end{equation*}

We follow the same derivation procedure as in~\autoref{sec:ipm_lb_spec_est}, so we proceed by including the details which differ from the previous section.
The Lagrangian of the dual barrier problem of~\eqref{eq:compressed_subprob_2} is as follows
\begin{equation}
\label{eq:dual_barrier_Lagrangian2}
\begin{aligned}
L_\mu(S, \eta, T, \zeta, \omega) &= \frac{1}{2} \,\svec(S)^\top Q_{11} \,\svec(S) + \eta \, q_{12}^\top \, \svec(S) \\ &\qquad+ \frac{1}{2} \eta^2 q_{22} + h_1^\top \svec(S) + \eta \, h_2 - \svec(S)^\top\svec(T) - \eta \zeta \\
&\qquad- \omega(1 - v_I^\top \svec(S) - \eta) + \mu (\log \det (T) + \log\zeta + \log \omega).
\end{aligned}
\end{equation}
The first-order optimality conditions of~\eqref{eq:dual_barrier_Lagrangian2} yields the following system of equations (after the same standard linearization)
\begin{equation}
    \label{eq:linearized_system2}
    \mathrm{F}_\mu(S, \eta, T, \zeta, \omega) := \begin{pmatrix}
        \vspace{0.1em} Q_{11}\, \svec(S) + \eta \,q_{12} + h_1 - \svec(T) + \omega v_I \\
        \vspace{0.1em} q_{12}^\top \, \svec(S) + \eta \,q_{22} + h_2 - \zeta + \omega \\
        \vspace{0.1em} ST - \mu I \\
        \vspace{0.1em} \eta\zeta - \mu\\
        \vspace{0.1em} \omega (1 - v_I^\top \svec(S) - \eta) - \mu\\
    \end{pmatrix}
    =: \begin{pmatrix}
        \vspace{0.1em} \mathsf{F}_1 \\
        \vspace{0.1em} \mathsf{F}_2 \\
        \vspace{0.1em} \mathsf{F}_3 \\
        \vspace{0.1em} \mathsf{F}_4 \\
        \vspace{0.1em} \mathsf{F}_5 \\
    \end{pmatrix}
    = 0.
\end{equation}
The Newton's method step direction $(\Delta S, \Delta \eta, \Delta T, \Delta \zeta, \Delta \omega)$ is determined via the following linearized system
\begin{align}
Q_{11} \svec(\Delta S) + \Delta \eta \,q_{12} - \svec(\Delta T) + \Delta \omega \, v_I &= - \mathsf{F}_1 \\
q_{12}^\top \, \svec(\Delta S) + \Delta \,\eta q_{22} - \Delta \zeta + \Delta \omega &= - \mathsf{F}_2 \\
\omega^{-1}(1 - v_I^\top \svec(S) - \eta) \Delta \omega - v_I^\top \svec(\Delta S) - \Delta \eta &= \mu \omega^{-1} + v_I^\top \svec(S) + \eta - 1 \\
(T \otimes_s S^{-1}) \, \svec(\Delta S) + \svec(\Delta T) &= \mu \,\svec(S^{-1}) - \svec(T) \\
\zeta \eta^{-1} \Delta \eta + \Delta \zeta &= \mu \eta^{-1} - \zeta
\end{align}
We can solve this system efficiently by analytically eliminating all variables except $\svec(\Delta S)$. We can compute $\svec(\Delta S)$ by solving the following linear matrix equation using an off-the-shelf linear system solver
\begin{equation}
\label{eq:delta_S_linear_system2}
\begin{aligned}
&\left( Q_{11} + T \otimes_s S^{-1} - (\kappa_1 \kappa_2 + 1)^{-1}\big( q_{12} (\kappa_1 q_{12} + v_I)^\top \!+ v_I(q_{12} - \kappa_2 v_I)^\top\big)\right)  \svec(\Delta S) \\
&\qquad=  q_{12} \left( (\kappa_1 \kappa_2 + 1)^{-1}\left( \mu \omega^{-1} + v_I^\top \svec(S) + \eta - 1 + \kappa_1 (\smsf{F}_2 - \mu \eta^{-1} + \zeta)\right)\right) \\
&\qquad\qquad + v_I\left((\kappa_1 \kappa_2 + 1)^{-1} \left( \smsf{F}_2 - \mu \eta^{-1} + \zeta - \kappa_2 (\mu \omega^{-1} + v_I^\top \svec(S) + \eta - 1) \right)\right) \\
&\qquad\qquad-\smsf{F}_1 + \mu \,\svec(S^{-1}) - \svec(T),
\end{aligned}
\end{equation}
where $\kappa_1 := \omega^{-1}(1 - v_I^\top \svec(S) - \eta)$ and $\kappa_2 := \zeta \eta^{-1} + q_{22}$.
Then, computing the rest of the update directions $\Delta \eta$, $\svec(\Delta T)$, $\Delta \zeta$, and $\Delta \omega$ amounts to back-substituting the solution to~\eqref{eq:delta_S_linear_system2} for $\svec(\Delta S)$ through the analytical variable elimination equations.

We make a single change to the initialization as compared with the interior point method presented in~\autoref{sec:ipm_lb_spec_est}. Since this interior point method can be executed multiple times in a row with only the value of $\tilde{\nu}$ changing as the first step in the alternating maximization algorithm, we warm-start initialize $\theta$ with the previous execution's $\theta_\star$. We observe non-negligible convergence improvements over arbitrary initialization as this interior point method procedure is called in sequence.

\subsection{Remark on Solving Subproblems}
\citet{ding2023revisiting} claim that~\eqref{eq:small_quad_sdp} could be solved using (accelerated) projected gradient descent and describe the method necessary for projection and proper scaling, but they use Mosek~\cite{aps2019mosek}, an off-the-shelf solver, in their experiments. We find that while theoretically possible, projected gradient descent does not work well in practice since choosing the correct step size to obtain a high quality solution to~\eqref{eq:small_quad_sdp} is difficult and varies between time steps $t$. We instead use (and advocate for) the primal-dual interior point method derived in Appendix~\ref{sec:small_quad_sdp_ipm}.

\section{Additional Details on Scaling with Matrix Sketching}
\label{sec:add_sketching}

%
The values $\langle C, \bar{X}_t \rangle$, $\tr(\bar{X}_t)$, and $\cA \bar{X}_t$ can be efficiently maintained given low-rank updates to $\bar{X}_t$ due to the linearity of the operations, 
\begin{align}
\langle C, \bar{X}_{t+1} &\rangle \gets \langle C, \eta_{t+1} \bar{X}_t + V_t Q_{\underline{c}} \Lambda_{\underline{c}} Q_{\underline{c}}^\top V_t^\top \rangle = \eta_{t+1} \langle C, \bar{X}_t \rangle + \tr\Big(V_t^\top Q_{\underline{c}}^\top\big(C V_t Q_{\underline{c}} \Lambda_{\underline{c}}\big)\Big),\\
\tr(\bar{X}_{t+1}) &\gets \tr\Big(\eta_{t+1} \bar{X}_t + V_t Q_{\underline{c}} \Lambda_{\underline{c}} Q_{\underline{c}}^\top V_t^\top\Big) = \eta_{t+1} \tr\big(\bar{X}_t\big) + \tr\!\left(\Lambda_{\underline{c}}\right), \\
\cA \bar{X}_{t+1} &\gets \cA \Big(\eta_{t+1} \bar{X}_t + V_t Q_{\underline{c}} \Lambda_{\underline{c}} Q_{\underline{c}}^\top V_t^\top\Big) = \eta_{t+1} \cA \bar{X}_t + \cA \big(V_t Q_{\underline{c}} \Lambda_{\underline{c}} Q_{\underline{c}}^\top V_t^\top \big),
\end{align}
where $\cA \big(V_t Q_{\underline{c}} \Lambda_{\underline{c}} Q_{\underline{c}}^\top V_t^\top \big)$ can be computed efficiently since 
\begin{equation}
\Big(\cA \big(V_t Q_{\underline{c}} \Lambda_{\underline{c}} Q_{\underline{c}}^\top V_t^\top \big)\Big)_i = \tr\Big(V_t^\top Q_{\underline{c}}^\top\big(A_i V_t Q_{\underline{c}} \Lambda_{\underline{c}}\big)\Big).
\end{equation}

\subsection{Nystr\"om Sketch Reconstruction}
Given the sketch $P_t$ and the projection matrix $\Psi$, we can compute 
a rank-$r$ approximation to $\bar{X}_t$. The approximation is as follows
\begin{equation}
    \label{eq:nystrom_approx}
    \widehat{\bar{X}_t} := P_t(\Psi^\top P_t)^+ P_t^\top = (\bar{X}_t \Psi) (\Psi^\top \bar{X}_t \Psi)^+(\bar{X}_t \Psi)^\top,
\end{equation}
where $\vphantom{\cdot}^+$ is the Moore-Penrose inverse. The reconstruction is called a \emph{Nystr\"om approximation}. We will almost always compute the best rank-$r$ approximation of $\widehat{\bar{X}_t}$ for memory efficiency. In practice, we use the numerically stable implementation of the Nystr\"om approximation provided by~\citet{yurtsever2021scalable}.
The Nystr\"om approximation yields a provably good approximation for any sketched matrix. See the following fact from~\cite{tropp2017fixed},
\begin{fact}
    Fix $X \in \SS^n_+$ arbitrarily. Let $P:= X \Psi$ where $\Psi \in \RR^{n \times r}$ has independently sampled standard normal entries.
    For each $r' < r - 1$, the Nystr\"om approximation $\widehat{X}$ as defined in~\eqref{eq:nystrom_approx} satisfies
    \begin{equation}
        \EE_\Psi \big\| X - \widehat{X}\big\|\nuc \leq \left( 1 + \frac{r'}{ r - r' - 1} \right) \big\| X - [\mkern-2.5mu[X]\mkern-2.5mu]_r \big\|\nuc,
    \end{equation}
    where $\EE_\Psi$ is the expectation with respect to $\Psi$ and $[\mkern-2.5mu[X]\mkern-2.5mu]_r$ is returns an $r$-truncated singular-value decomposition of the matrix $X$, which is a best rank-r approximation with respect to every unitarily invariant norm\emph{~\cite{mirsky1960symmetric}}.
    If we replace $\widehat{X}$ with $[\mkern-2.5mu[\widehat{X}]\mkern-2.5mu]_r$, this error bound still remains valid.
\end{fact}

\subsection{USBS Memory Requirements.}
When we implement USBS using this matrix sketching procedure we see significant decrease in time and memory required by USBS. Storing the projection matrix $\Psi$ and $P_t$ requires $\cO(nr)$ memory. Storing $\langle C, \bar{X}_t \rangle$ and $\tr(\bar{X}_t)$ requires $\cO(1)$ memory, respectively, and storing $\cA \bar{X}_t$, $\nu_t$, $y_t$, and $\tilde{y}_{t+1}$ requires $\cO(m)$ memory. The primal-dual interior point methods require $\cO(k^2)$ and $\cO(k^4)$ memory, respectively, and storing $V_t$ requires $\cO(nk)$ memory.
This means that the entire memory required by USBS is $\cO(nr + nk + m + k^4)$ working memory (not including the memory to store the problem data) which for many SDPs is much less than explicitly storing the iterate $\bar{X}_t$ which requires $\cO(n^2)$ memory.

\section{Proof of Theorem~\ref{thm:sublinear_convergence}}
\label{sec:main_thm_proof}

In this section, we present the non-asymptotic convergence results for USBS. 
 First, we summarize the relevant results presented by \citet{diaz2023optimal} in the following theorem.
\begin{theorem}
\label{thm:proximal_bundle_convergence}
Let $\beta > 0$ and $\rho > 0$ be constant, $f: \RR^m \to (-\infty, +\infty]$ be a proper closed convex function with nonempty set of minimizers $\cY_\star$, and the sequence of models produced by the proximal bundle method $\left\{ \hat{f}_{t+1}: \RR^m \to (-\infty, +\infty] \right\}$ satisfy the conditions~\eqref{eq:minorant_cond},~\eqref{eq:obj_subgrad_cond}, and \eqref{eq:model_subgrad_cond}. If the iterates $y_t$ and subgradients $g_{t+1}$ are bounded,
then the iterates $y_t$ have $f(y_t) - f(y_\star) \leq \varepsilon$ for all $t \geq \cO(1 / \varepsilon^3)$. Additionally, if $f(y) - f(y_\star) \geq \cO(\dist^2(y, \cY_\star))$ for all $y$, then the bound improves to $f(y_t) - f(y_\star) \leq \varepsilon$ for all $t \geq \cO(1 / \varepsilon)$.
\end{theorem}
\noindent We will use this result to prove the non-asymptotic convergence rates for the spectral bundle method.

Note the following fact necessary to use Theorem~\ref{thm:proximal_bundle_convergence}.

\begin{fact}
    \label{fact:bounded}
    The set $\{y \in \cY : f(y_\star) \leq f(y) \leq f(y_0)\}$ is compact and contains all iterates $y_t$.
    Moreover, the iterates $y_t$ and subgradients $g_{t+1}$ are bounded, i.e.\! $\sup_{\,t \geq 0} \{\dist(y_t, \cY_*)\} < \infty$ and $\sup_{\,t \geq 0} \left\{\|g_{t +1 }\| : g_{t+1} \in \partial f(\tilde{y}_{t+1})\right\} < \infty$.
\end{fact}

\begin{proof}
The function $f$ is continuous and proper over the domain $\cY$.
It is a well-known fact that the level sets of continuous proper functions are compact. 
Hence, the union of level sets $\{y \in \cY : f(y_\star) \leq f(y) \leq f(y_0)\}$ is compact, and therefore, $\sup_{\,t \geq 0} \{\dist(y_t, \cY_*)\} < \infty$.
The subgradients take the form $g_{t+1} = b + \alpha \cA(v v^{\!\top}) \in \partial f(\tilde{y}_{t+1})$ where $v \in \RR^n$ is a unit-normed vector, and thus, $\|g_{t+1}\| < \infty$ for all $t$. 
\end{proof}

\noindent The following lemma guarantees quadratic growth whenever strict complementarity holds.

\begin{lemma}
    \label{lem:error_bound}
    Let $y \in \cY$ be arbitrary. Then,
    \begin{equation}
        \dist^{2^{d}}(y, \cY_\star) \leq \cO(f(y) - f(y_\star)),
    \end{equation}
    where $d \in \{0, 1, \ldots, m\}$ is the singularity degree\emph{~\cite{sturm2000error,drusvyatskiy2017many,sremac2021error}}. If strict complementarity holds, then $d \leq 1$, and if Slater's condition holds, then $d = 0$.
\end{lemma}

\begin{proof}
Let $\cF := \{y \in \RR^m : \cA\adj y - C \in \SS^n_+\}$. Observe that $\cY_\star = \cY \cap \cF$.
Since $\cY_\star$ is compact, the H\"olderian error bound~\cite{sturm2000error, drusvyatskiy2017many, sremac2021error} ensures that
\begin{equation*}
\dist(y, \cY_\star) \leq \cO \left(\dist^{2^{-d}}(y, \cY) + \dist^{2^{-d}}(y, \cF)\right),
\end{equation*}
where $d \in \{0, 1, \ldots, m\}$ is the singularity degree of the SDP.
This implies that for $y \in Y$,
\begin{equation*}
\begin{aligned}
\dist^{2^{d}}(y, \cY_\star) &\leq \cO \left(\dist(y, \cF)\right) \\
&\leq \cO \left( \alpha \max \{ \lambda_\textrm{max}(C - \cA\adj y), 0\} \right) \\
&\leq \cO \left( f(y) - f(y_\star) \right).
\end{aligned}
\end{equation*}
\end{proof}

\noindent The following three lemmas guarantee primal feasiblity, dual feasibility, and primal-dual objective optimality given convergence of the penalized dual objective, i.e. $f(y_t) - f(y_\star) \leq \varepsilon$.

\begin{lemma}
\label{lem:primal_feasibility}
At every descent step $t$,
\begin{equation}
    \dist(\cA X_{t+1}, \cK) \leq \cO\!\left(\sqrt{f(y_t) - f(y_\star)}\right).
\end{equation}
Additionally, if Slater's condition holds and $y_\star$ is unique, then
\begin{equation}
    \dist(\cA X_{t+1}, \cK) \leq \cO(f(y_t) - f(y_\star)).
\end{equation}
\end{lemma}

\begin{proof}
For any descent step $t$, it is easy to verify from~\eqref{eq:dual_update} and~\eqref{eq:dual_slack_update} that 
\begin{equation*}
    \proj_\cK(\cA X_{t+1}) - \cA X_{t+1} \leq \rho (y_t - y_{t+1}),
\end{equation*}
and therefore,
\begin{equation*}
    \| \cA X_{t+1} -  \proj_\cK(\cA X_{t+1})\|^2 \leq \rho^2 \| y_t - y_{t+1} \|^2.
\end{equation*}
To complete the proof we utilize the fact that $\hat{f}_t(y_{t+1}) = \min_{y \in \cY} \hat{f}_t(y) + \frac{\rho}{2}\|y - y_t\|^2$, the fact that $\hat{f}_t(y) \leq f(y)$ for all $y \in \cY$, and the definition of a descent step, which gives
\begin{equation*}
    \frac{\rho}{2} \|y_{t+1} - y_t\|^2 \leq \hat{f}_t(y_t) - \hat{f}_t(y_{t+1}) \leq f(y_t) - \hat{f}_t(y_{t+1}) \leq \frac{f(y_t) - f(y_{t+1})}{\beta} \leq \frac{f(y_t) - f(y_\star)}{\beta}.
\end{equation*}
Assume Slater's condition holds and $y_\star$ is unique. Then, using Lemma~\ref{lem:error_bound} gives
\begin{equation*}
\begin{aligned}
    \| \cA X_{t+1} -  \proj_\cK(\cA X_{t+1})\|^2 &\leq \rho^2 \| y_t - y_{t+1} \|^2\\
    &\leq \rho^2 \left(\|y_t - y_\star\|^2 + \|y_{t+1} - y_\star\|^2\right) \\
    &\leq \cO\!\left((f(y_t) - f(y_\star))^2\right) + \cO\!\left((f(y_{t+1}) - f(y_\star))^2\right) \\
    &\leq  \cO\!\left((f(y_t) - f(y_\star))^2\right).
\end{aligned}
\end{equation*}
\end{proof}

\begin{lemma}
\label{lem:dual_feasibility}
Suppose strong duality holds and $\alpha \geq 2 \mathfrak{N}(\cX_\star)$. Then, at every descent step $t$,
\begin{equation}
\lambda_{\max}(C - \cA\adj y_t) \leq \cO\!\left(f(y_t) - f(y_\star)\right).
\end{equation}
\end{lemma}

\begin{proof}
By definition of strong duality, we know that for any $X_\star \in \cX_\star$ that $\langle C, X_\star \rangle = \langle b, y_\star \rangle$, and equivalently, $\langle X_\star, \cA\adj y_\star - C \rangle = 0$.
Then, the dual objective gap is bounded as follows
\begin{align*}
\langle b, y_t - y_\star \rangle &= \langle \cA X_\star, y_t - y_\star \rangle \\
&= \langle X_\star, \cA\adj(y_t - y_\star)\rangle \\
&= \langle X_\star, (\cA\adj y_t - C) - (\cA\adj y_\star - C) \rangle \\
&= \langle X_\star, \cA\adj y_t - C \rangle \\
&\geq - \| X_\star \|\nuc \max\{\lambda_{\max}(C - \cA\adj y_t), 0\}.
\end{align*}
We now utilize this bound to obtain the desired result
\begin{equation*}
f(y_t) - f(y_\star) = \langle b, y_t - y_\star \rangle + \alpha \max\{ \lambda_{\max} (C - \cA\adj y_t), 0\} \geq \| X_\star \|\nuc \max\{ \lambda_{\max}(C - \cA\adj y_t), 0\},
\end{equation*}
where the last inequality is implied by the assumption that $\alpha \geq 2 \mathfrak{N}(\cX_\star) \geq 2 \|X_\star\|\nuc$.

\end{proof}

\begin{lemma}
\label{lem:primal_dual_opt}
At every descent step $t$,
\begin{equation}
\left|\langle b, y_{t+1} \rangle - \langle C, X_{t+1} \rangle \right| \leq \cO\!\left(f(y_t) - f(y_\star)\right) + \cO\!\left(\sqrt{f(y_t) - f(y_\star)}\right).
\end{equation}
Additionally, if Slater's condition holds and $y_\star$ is unique, then
\begin{equation}
   \left|\langle b, y_{t+1} \rangle - \langle C, X_{t+1} \rangle \right| \leq \cO(f(y_t) - f(y_\star)).
\end{equation}
\end{lemma}

\begin{proof}
We start by rewriting the primal-dual gap as follows
\begin{align*}
    \langle C, X_{t+1} \rangle - \langle b, y_{t+1} \rangle
   &= \langle C, X_{t+1} \rangle - \langle \cA X_{t+1}, y_{t+1}\rangle  + \langle \cA X_{t+1} - b, y_{t+1} \rangle \\
   &= \langle X_{t+1}, C - \cA\adj y_{t+1}\rangle + \langle \cA X_{t+1} - b, y_{t+1} \rangle. 
\end{align*}
We will now bound the absolute values of the two resulting terms to bound the desired quantity.
Using the Cauchy-Schwarz inequality and Lemma~\ref{lem:primal_feasibility} it can be seen
\begin{align*}
    |\langle \cA X_{t+1} - b, y_{t+1} \rangle |
    &\leq \|\cA X_{t+1} - b\| \|y_{t+1}\| \\
    &\leq \cO\!\left(\|y_{t+1}\| \sqrt{f(y_t) - f(y_\star)} \right) \\
    &\leq \cO\!\left(\sqrt{f(y_t) - f(y_\star)} \right),
\end{align*}
where the last inequality comes from Fact~\ref{fact:bounded}.
Assume Slater's condition holds and $y_\star$ is unique. Then, using Lemma~\ref{lem:error_bound} gives
\begin{equation*}
    |\langle \cA X_{t+1} - b, y_{t+1} \rangle |\leq \cO\left(f(y_t) - f(y_\star) \right).
\end{equation*}
Since $X_{t+1} \succeq 0$ and $\tr(X_{t+1}) \leq \alpha$ by construction, we can use Lemma~\ref{lem:dual_feasibility} to yield
\begin{align*}
    |\langle X_{t+1}, C - \cA\adj y_{t+1} \rangle |
    &\leq \cO(\max\{C - \cA\adj y_{t+1}, 0 \}) \\
    &\leq \cO(f(y_{t+1}) - f(y_\star)) \\
    &\leq \cO(f(y_t) - f(y_\star)).
\end{align*}
The immediately yields the desired result.
\end{proof}


\subsection{Proof of Theorem~\ref{thm:sublinear_convergence}}
\label{sec:sublinear_convergence_proof}

The overall proof strategy is to use Theorem~\ref{thm:proximal_bundle_convergence} to show that the penalized dual gap, $f(y_t) - f(y_*)$, converges at a worst case rate of $\cO(1 / \varepsilon^3)$, which improves to $\cO(1 / \varepsilon)$ if Slater's condition or strict complementarity hold.
Then, we can use Lemmas~\ref{lem:primal_feasibility}, ~\ref{lem:dual_feasibility}, and~\ref{lem:primal_dual_opt} to obtain the stated convergence of primal feasibility, dual feasibility, and primal-dual optimality.

To apply the result of Theorem~\ref{thm:proximal_bundle_convergence}, we showed  that the norms of the iterates and subgradients of $f$ are bounded (Fact~\ref{fact:bounded}), and so all we need to do is show that the model satisfies the conditions~\eqref{eq:minorant_cond},~\eqref{eq:obj_subgrad_cond}, and \eqref{eq:model_subgrad_cond}.

Let $v$ be a maximum eigenvector $C - \cA\adj \tilde{y}_{t+1}$ \thinspace if \thinspace $\lambda_{\max}(C - \cA\adj \tilde{y}_{t+1}) > 0$ and $v = 0$, otherwise.
Then denote $g_{t+1} = b + \alpha \cA(vv^\top) \in \partial f(\tilde{y}_{t+1})$ as the subgradient of $f$ corresponding $v$ at the candidate iterate $\tilde{y}_{t+1}$.
Let $s_{t+1} = \rho (y_t - \tilde{y}_{t+1}) \in \partial \hat{f}_t(\tilde{y}_{t+1})$ be the aggregate subgradient.
Recall that at iteration $t+1$ the model approximates $\cX$ by the low-dimensional spectral set
\begin{equation*}
\widehat{\cX}_{t+1} := \left\{\eta \bar{X}_{t+1} + V_{t+1} S V_{t+1}^\top :  \eta\, \tr(\bar{X}_{t+1}) + \tr(S) \leq \alpha, \eta \geq 0, S \in \SS^k_+\right\}.
\end{equation*}
and therefore the penalized dual objective and model respectively take the following forms for any $y \in \cY$,
\begin{equation}
    \label{eq:obj_and_model_feasible}
    f(y) = \sup_{X \in \cX} \langle C - \cA\adj y , X\rangle + \langle b, y \rangle \quad \textrm{and} \quad \hat{f}_{t+1}(y) = \sup_{X \in \widehat{\cX}_{t+1}} \langle C - \cA\adj y , X\rangle + \langle b, y \rangle.
\end{equation}

\paragraph{Verifying~\eqref{eq:minorant_cond}.} The condition~\eqref{eq:minorant_cond} follows immediately from~\eqref{eq:obj_and_model_feasible}, since $\widehat{\cX}_{t+1} \subseteq \cX$.

\paragraph{Verifying~\eqref{eq:obj_subgrad_cond}.} Since $V_{t+1}$ spans $v$, there exists a vector $s \in \RR^k$ such that $V_{t+1}s = v$. Taking $\eta = 0$ and $S = \alpha ss^{\!\top}$ implies $\alpha vv^{\!\top} \in \widehat{\cX}_{t+1}$. Hence, 
\begin{equation*}
    \begin{aligned}
    \hat{f}_{t+1}(y) &\geq \langle C - \cA\adj y, \alpha v v^{\!\top} \rangle + \langle b, y \rangle \\
    &= \langle C - \cA\adj, \alpha v v^{\!\top} \rangle + \langle b, \tilde{y}_{t+1} \rangle + \langle b + \alpha \cA (v v^{\!\top}), y - \tilde{y}_{t+1}\rangle \\
    &= f(\tilde{y}_{t+1}) + \langle g_{t+1}, y - \tilde{y}_{t+1} \rangle.
    \end{aligned}
\end{equation*}

\paragraph{Verifying~\eqref{eq:model_subgrad_cond}.} From the first-order optimality conditions and the update step~\eqref{eq:dual_update} we know that
\begin{equation*}
\begin{aligned}
    s_{t+1} &= \rho(y_t - \tilde{y}_{t+1}) = b - \nu_{t+1} - \cA X_{t+1},  \\
    \hat{f}_t(\tilde{y}_{t+1}) &= \langle C - \cA\adj \tilde{y}_{t+1}, X_{t+1} \rangle + \langle b - \nu_{t+1}, \tilde{y}_{t+1} \rangle.
\end{aligned}
\end{equation*}
To show the desired result, we first want to show that $X_{t+1} \in \widehat{X}_{t+1}$. If $k_p = 0$, we are done since $X_{t+1} = \bar{X}_{t+1}$.
Otherwise, $k_p \geq 1$. First note that $\tr(X_{t+1}) \leq \alpha$ by construction.
Then,
\begin{align*}
    X_{t+1} 
    &= \eta_{t+1} \bar{X}_t + V_t S_{t+1} V_t^\top \\
    &= \eta_{t+1} \bar{X}_t + V_t (Q_{\overline{p}} \Lambda_{\overline{p}}  Q_{\overline{p}}^\top + Q_{\underline{c}} \Lambda_{\underline{c}}  Q_{\underline{c}}^\top) V_t^\top \\
    &= \eta_{t+1} \bar{X}_t + V_t Q_{\overline{p}} \Lambda_{\overline{p}}  Q_{\overline{p}}^\top V_t^\top + V_t Q_{\underline{c}} \Lambda_{\underline{c}}  Q_{\underline{c}}^\top V_t^\top \\
    &= \bar{X}_{t+1} +  V_t Q_{\overline{p}} \Lambda_{\overline{p}}  Q_{\overline{p}}^\top V_t^\top.
\end{align*}
Since $V_{t+1}$ spans each of the columns of $V_t Q_{\overline{p}}$, we can conclude that $X_{t+1} \in \widehat{X}_{t+1}$. Thus, for any $y \in \cY$,
\begin{equation*}
    \hat{f}_{t+1}(y) \geq \langle C - \cA\adj y, X_{t+1} \rangle + \langle b - \nu_{t+1}, y \rangle = \hat{f}_t (\tilde{y}_{t+1}) + \langle s_{t+1}, y - \tilde{y}_{t+1}\rangle.
\end{equation*}

\section{Experimental Setup and Details}
\label{sec:add_expt_setup}

Following~\cite{yurtsever2021scalable}, we scale the problem data for all problems such that the following condition holds,
\begin{equation}
    \label{eq:std_scaling}
    \| C \|\fro = \tr(X_\star) = 1.
\end{equation}
In all three problem types, the constraints exactly determine the trace of all optimal solutions.
Since we know the trace of the optimal solution and we apply problem scaling~\eqref{eq:std_scaling}, we can set $\alpha = 2$ for all problem types.

Computing the relative infeasibility is simple. As for the relative primal objective suboptimality, we do not usually have access to the optimal primal objective value $\langle C, X_\star \rangle$, but we can compute an upper bound. For any $y \in \cY$, 
\begin{equation}
\label{eq:objective_subopt_upper_bound}
\begin{aligned}
&f(y) = \alpha \max\{\lambda_\textrm{max}(C - \cA\adj y), 0\} + \langle b, y \rangle \geq \langle C - \cA\adj y, X_\star \rangle - \langle b, y \rangle = \langle C, X_\star \rangle, \\
& \quad \implies \langle C, X \rangle - \langle C, X_\star \rangle \leq \langle C, X \rangle - f(y).
\end{aligned}
\end{equation}
We use this upper bound to approximate the relative primal objective suboptimality in our experiments for USBS.
We use the upper bound given in~\cite{yurtsever2021scalable} to compute the relative primal objective suboptimality for CGAL.

For both CGAL and USBS, we compute eigenvectors and eigenvalues using an implementation of the thick-restart Lanczos algorithm~\cite{wu1999thick,hernandez2007krylov} with 32 inner iterations and a maximum of 10 restarts.
Both CGAL and USBS are implemented using 64-bit floating point arithmetic.
Unless otherwise specified, we use a CPU machine with 16 cores and up to 128 GB of RAM.
Throughout the presentation of the experimental results we use $\uparrow$ to indicate that \emph{higher is better} for that particular metric and $\downarrow$ to indicate \emph{lower is better} for that particular metric.

\subsection{MaxCut}
The MaxCut problem is a fundamental combinatorial optimization problem and its SDP relaxation is a common test bed for SDP solvers.
Given an undirected graph, the MaxCut is a partitioning of the $n$ vertices into two sets such that the number of edges between the two subsets is maximized.
Formally, the MaxCut problem can be written as follows
\begin{equation}
    \label{eq:maxcut}
    \maximize \quad \frac{1}{4} \,x^\top\!L\,x \quad \st \quad x \in \{\pm 1\}^n,
\end{equation}
where $L$ is the graph Laplacian. This optimization problem is known to be NP-hard~\cite{karp1972}, but rounding the solution to the SDP relaxation can yield very strong approximate solutions~\cite{goemans1995improved}.

\subsubsection{SDP Relaxation}
For any vector $x \in \{\pm 1\}^n$, the matrix $X := xx^\top$ is positive semidefinite and all of its diagonal entries are exactly one. 
These implicit constraints allow us to arrive at the MaxCut SDP as follows
\begin{equation}
    \maximize \quad \frac{1}{4} \, \tr(L X) \quad \st \quad \textrm{diag}(X) = \mathbf{1} \textrm{ and } X \succeq 0,
\end{equation}
where $\textrm{diag}(\cdot)$ extracts the diagonal of a matrix into a vector and $\textbf{1} \in \RR^{n}$ is the all-ones vector.
A matrix solution $X_\star$ to~\eqref{eq:maxcut_sdp} does not readily elicit a graph cut. 
One way to round $X_\star$ is to compute a best rank-one approximation $x^{}_\star x_\star^\top$ and use $\textrm{sgn}(x_\star)$ as an approximate cut.
There are more sophisticated rounding procedures (e.g.~\cite{goemans1995improved}), but this one works quite well in practice.

Notice that~\eqref{eq:maxcut_sdp} contains exactly $n$ equality constraints. 
Additionally, in many instances of MaxCut, the graph Laplacian $L$ is sparse and the optimal solution is low-rank. 
This means that we can represent the MaxCut instance with far less than $\cO(n^2)$ memory.
These attributes make sketching the matrix variable $X$ with relatively small $r$ an attractive and effective approach when considering the MaxCut problem.

\subsubsection{Rounding}
Both CGAL and USBS maintain a sketch of the primal variable that, along with the test matrix, can be used to construct a low-rank approximation of the primal iterate $\widehat{X} = U \Lambda U^\top$ where $U \in \RR^{n \times r}$ has orthonormal columns.
Following~\cite{yurtsever2021scalable}, we evaluate the size of $r$ cuts given by the column vectors of $\textrm{sgn}(U)$ and choose the largest.

\subsubsection{Experimental Setup}

We evaluate the CGAL and USBS with and without warm-starting on ten instances from the DIMACS10~\cite{dimacs10} where $n$ and the number of edges in each graph are shown in~\autoref{tab:maxcut_dataset_stats}.
We warm-start each method by dropping the last 1\% of vertices from the graph resulting in a graph with 99\% of the vertices.
We pad the solution from the 99\%-sized problem with zeros and rescale as necessary to create the warm-start initialization.
Unless otherwise specified, we use the following hyperparameters for USBS in our experiments: $r=10$, $\rho=0.01$, $\beta=0.25$, $k_c = 10$, $k_p = 1$. In our experiments, we find that $k_p > 0$ does improve performance slightly.
See the implementation for more details.



\subsection{Quadratic Assignment Problem}

The quadratic assignment problem (QAP) is a very difficult but fundamental class of combinatorial optimization problems containing the traveling salesman problem, max-clique, bandwidth problems, and facilities location problems among others~\cite{loiola2007survey}. 
SDP relaxations have been shown to facilitate finding good solutions to large QAPs~\cite{zhao1998semidefinite}. 
The QAP can be formally defined as follows
\begin{equation}
    \label{eq:qap}
    \minimize \quad \tr\left(W \Pi D \Pi^\top\right) \quad \st \quad \Pi \textrm{ is an } \smtt{n} \times \smtt{n} \textrm{ permutation matrix},
\end{equation}
where $W \in \SS^\smtt{n}$ is the weight matrix, $D \in \SS^\smtt{n}$ is the distance matrix, and the goal is to optimize for an assignment $\Pi$ which aligns $W$ and $D$.
The number of $\smtt{n} \times \smtt{n}$ permutation matrices is $\smtt{n}!$, so a brute-force search becomes quickly intractable as \smtt{n} grows. Generally, QAP instances with $\smtt{n} > 30$ are intractable to solve exactly. In our experiments, using an SDP relaxation and rounding procedure, we can obtain good solutions to QAPs where \smtt{n} is between 136 and 198.
 
\subsubsection{SDP Relaxation}
There are many SDP relaxtions for QAPs, but we consider the one presented by~\cite{yurtsever2021scalable} and inspired by~\cite{huang2014scalable, bravo2018semidefinite} which is formulated as follows
\begin{equation}
    \begin{aligned}
        &\minimize  \quad \tr \big( (D \otimes W) \, \smsf{Y} \big) \\
        &\;\st \;\;\, \tr_1(\smsf{Y}) = I, \; \tr_2(\smsf{Y}) = I, \; \mathcal{G}(\smsf{Y}) \geq 0, \\
        &\qquad\;\; \textrm{vec}(\smsf{B}) = \textrm{diag}(\smsf{Y}), \; \smsf{B} \textbf{1} = \textbf{1}, \; \textbf{1}^\top \smsf{B} = \textbf{1}^\top, \; \smsf{B} \geq 0, \\
        &\qquad\;\; X := \begin{bmatrix}
            1 & \textrm{vec}(\smsf{B})^\top \\
            \textrm{vec}(\smsf{B}) & \smsf{Y}
        \end{bmatrix} \succeq 0, \; \tr(\smsf{Y}) = \smtt{n}
    \end{aligned}
\end{equation}
where $\otimes$ denotes the Kronecker product, $\tr_1(\cdot)$ and $\tr_2(\cdot)$ denote the partial trace over the first and second systems of Kronecker product respectively, $\mathcal{G}(\smsf{Y})$ extracts the entries of $\smsf{Y}$ corresponding to the nonzero entries of $D \otimes W$, and $\textrm{vec}(\cdot)$ stacks the columns of a matrix one on top of the other to form a vector.
In many cases, one of $D$ or $W$ is sparse (i.e.~$\mathcal{O}(n)$ nonzero entries) resulting in $\mathcal{O}(\smtt{n}^3)$ total constraints for the SDP.

The primal variable $X$ has dimension $(\smtt{n}^2 + 1) \times (\smtt{n}^2 + 1)$ and as a result the SDP relaxation has $\mathcal{O}(\smtt{n}^4)$ decision variables.
Given the aggressive growth in complexity, most SDP based algorithms have difficulty operating on instances where $\smtt{n} > 50$.
To reduce the number of decision variables, we sketch $X$ with $r = \smtt{n}$, resulting in $\mathcal{O}(\smtt{n}^3)$ entries in the sketch and the same number of constraints in most natural instances.
We show that this enables CGAL and USBS to scale to instances where $\smtt{n} = 198$ (more than 1.5 billion decision variables and constraints).

\subsubsection{Rounding}
We adopt the rounding method presented in~\cite{yurtsever2021scalable} to convert an approximate solution of~\eqref{eq:qap_sdp} into a permutation matrix.
Reconstructing an approximation to the primal variable from the sketch yields $\widehat{X} = U \Lambda U^\top\!$~where $U$ has shape $(\smtt{n}^2 + 1) \times \smtt{n}$. For each column of $U$, we discard the first entry of the column vector, reshape the remaining $\smtt{n}^2$ entries into an $\smtt{n} \times \smtt{n}$ matrix, and project the reshaped matrix onto the set of $\smtt{n} \times \smtt{n}$ permutation matrices using the Hungarian method (also known as Munkres' assignment algorithm)~\cite{kuhn1955hungarian,munkres1957algorithms, jonker1988shortest}.
This procedure yields a feasible point to the original QAP~\eqref{eq:qap}.
To get an upper bound for~\eqref{eq:qap}, we perform the rounding procedure on each column of $U$ and choose the permutation matrix $\Pi$ which minimizes the objective $\tr\!\left(W \Pi D \Pi^\top\right)$.

\subsubsection{Experimental Setup}
We evaluate CGAL and USBS on select large instances from \smtt{QAPLIB}~\cite{burkard1997qaplib} and \smtt{TSPLIB}~\cite{reinelt1995tsplib95} ranging in size from 136 to 198.
Provided with each of these QAP instances is the known optimum.
For many instances, we find that the quality of the permutation matrix produced by the rounding procedure does not entirely correlate with the quality of the iterates with respect to the SDP~\eqref{eq:qap_sdp}.
Thus, we apply the rounding procedure at every iteration of both CGAL and USBS.
Our metric for evaluation is \smtt{relative gap} which is computed as follows
\begin{equation*}
    \smtt{relative gap} = \frac{\smtt{upper bound obtained} - \smtt{optimum}}{\smtt{optimum}}.
\end{equation*}
We report the lowest value so far of \smtt{relative gap} as \smtt{best relative gap}.
CGAL~\cite{yurtsever2021scalable} is shown to obtain significantly smaller \smtt{best relative gap} than CSDP~\cite{bravo2018semidefinite} and PATH~\cite{zaslavskiy2008path} on most instances, and thus, is a strong baseline.

To create a warm-start initialization, we create a slightly smaller QAP by dropping the final row and column of both $D$ and $W$ (i.e. solve a size $\smtt{n} - 1$ subproblem of the original instance).
We use the solution to slightly smaller problem to set a warm-start initialization for the original problem, rescaling and padding with zeros where necessary.
We stopped optimizing after one hour. When warm-starting, we optimize the slightly smaller problem for one hour and then optimize the original problem for one hour.

Following~\cite{yurtsever2021scalable}, we apply the following scaling of the problem variables in~\eqref{eq:qap_sdp}
\begin{equation}
    \| C \|\fro = \tr(X_\star) = \|\cA\|\op = 1 \quad \textrm{and} \quad \| A_1 \|\fro = \cdots = \|A_m\|\fro.
\end{equation}
We use the following hyperparameters for USBS in our experiments: $\rho=0.005$, $\beta=0.25$, $k_c = 2$, $k_p = 0$. In our experiments, we find that $k_p > 0$ does not improve performance.
See the implementation for more details.

\subsection{Interactive Entity Resolution with $\exists$-constraints}
\label{sec:add_ecc}
\citet{pmlr-v162-angell22a} introduced a novel SDP relaxation of a combinatorial optimization problem that arises in an automatic knowledge base construction problem~\citep{uhlen2010towards,krishnamurthy2015learning,iv2022fruit}.
They consider the problem of interactive entity resolution with user feedback to correct predictions made by a machine learning algorithm.
Entity resolution (also known as coreference resolution or record linkage) is the process of identifying which mentions (sometimes referred to as records) of entities refer to the same entity~\cite{binette2022almost, angell2020clustering, agarwal2021entity, agarwal2022entity, yadav2021event}.
Entity resolution is an important problem central to automated knowledge base construction where the correctness of the resulting knowledge base is vital to its usefulness.
The entity resolution problem is fundamentally a clustering problem where the prefect clusters contain all mentions of exactly one entity.
Due to the large volumes of raw data frequently involved, machine learning approaches are often used to perform entity resolution.
To correct inevitable errors in the predictions made by these automated methods, human feedback can be used to correct errors in entity resolution.

Utilizing human feedback to guide and correct clustering decisions can be broadly categorized into two groups~\cite{bae2020interactive}: active clustering (also know as semi-supervised clustering)~\cite{viswanathan2023large, vikram2016interactive, mazumdar2017clustering, sato2009interactive, xiong2016active} and interactive clustering~\cite{chuang2014human, basu2010assisting, coden2017method, dubey2010cluster}\footnote{Commonly, \emph{active} clustering is used to describe techniques where the algorithm queries the human for feedback about a specific pair or group of points or clusters (akin to active learning) and \emph{interactive} clustering is used to describe techniques where humans observe the clustering output and provide unelicited feedback, typically in the form of some type of constraint, to the algorithm.}.
There are several paradigms of human feedback used to steer a clustering algorithm~\cite{frome2007learning, wagstaff2001constrained,shental2004computing,klein2002instance, kulis2009semi,lu2008constrained,li2008pairwise,li2009constrained,xing2002distance}, the most common of which is pairwise constraints~\cite{wagstaff2000clustering} (i.e.~statements about whether two points or mentions must or cannot be clustered together).
Recently, \citet{pmlr-v162-angell22a} introduced \emph{\econstraints\!\!}, a new paradigm of interactive feedback for correcting entity resolution predictions and presented a novel SDP relaxation as part of a heuristic algorithm for satisfying \econstraints in the predicted clustering.
This novel form of feedback allows users to specify constraints stating the existence of an entity with and without certain features.

\subsubsection{\econstraints}
To define \econstraints precisely, we must introduce some terminology and notation. 
Let the set of mentions be represented by the vertex set $V$ of a fully connected graph. We assume that each of the mentions $v \in V$ has an associated set of discrete features (attributes, relations, other features, etc.)~that is readily available or can be extracted easily using some automated method.
Let $\Phi = \{\phi_1, \phi_2, \ldots\}$ be the set of possible features and $\Phi(v) \subseteq \Phi$ be the subset of features associated with vertex $v \in V$.
A clustering $\cC = \{C_1, C_2, \ldots, C_m\}$ of vertices $V$ is a set of nonempty sets such that $C_i \subseteq V$, $C_i \cap C_j = \emptyset$ for all $i \ne j$, and $\bigcup_{C \in \cC} C = V$.
We denote the ground truth clustering of the vertices (partitioning mentions into entities) as $\cC^\star$.
Furthermore, for any subset of vertices $S \subseteq V$, we define the set of features associated with $S$ as $\Phi(S) = \bigcup_{v \in S} \Phi(v)$. 
For any cluster $C_i$, we use $\Phi(C_i)$ as the canonicalization function for a predicted entity, but this framework allows for more complex (and even learned) canonicalization functions.
When providing feedback, users are able to view the features of predicted clusters and then provide one or more \econstraints\!\!.

Formally, a \econstraint $\xi \subseteq \{+, -\} \times \Phi$ uniquely characterizes a constraint asserting there exists a cluster $C \in \cC^\star$ such that all of the positive features $\xi^{+} := \{\phi \in \Phi: (+, \phi) \in \xi\} $ are contained in the features of $C$ (i.e. $\xi^{+}\subseteq \Phi(C)$) and none of the negative features $\xi^{-} := \{\phi \in \Phi: (-, \phi) \in \xi\}$ are contained in the features of $C$ (i.e. $\xi^{-} \cap \Phi(C) = \emptyset$).
At any given time, let the set of \econstraints be represented by $\Xi$.
We say that a subset of nodes $S \subseteq V$ \emph{satisfies} a \econstraint $\xi$ if $\xi^{+}\subseteq \Phi(S)$ and $\xi^{-} \cap \Phi(S) = \emptyset$.
Note that more than one ground truth cluster can satisfy a \econstraint\!\!.
We say that a subset of vertices $S \subseteq V$ is \emph{incompatible} with a \econstraint $\xi$ if $\Phi(S) \cap \xi^{-} \ne \emptyset$ and denote this as $S \perp \xi$. Similarly, two \econstraints $\xi_a, \xi_b$ are incompatible if $\xi_a^+ \cap \xi_b^{-} \ne \emptyset$ or $\xi_a^- \cap \xi_b^{+} \ne \emptyset$, and is also denoted $\xi_a \perp \xi_b$. Furthermore, we say a subset of vertices $S \subseteq V$ and a \econstraint $\xi$ or two \econstraints $\xi_a, \xi_b$ are \emph{compatible} if they are not incompatible and denote this as $S \not\perp \xi$ and $\xi_a \not \perp \xi_b$, respectively.
Moreover, we use \emph{compatible} to describe when a subset of vertices $S \subseteq V$ could be part of a cluster that \emph{satisfies} a certain $\exists$-constraint $\xi$. This definition allows for $\Phi(S) \cap \xi^+$ to be empty, but still enables $S$ to be compatible with $\xi$ just as long as $\Phi(S) \cap \xi^-$ is empty.
Observe that if every mention has a unique feature, the \econstraint framework fully encompasses must-link and cannot-link constraints.

\subsubsection{SDP relaxation}
Following~\cite{pmlr-v162-angell22a}, the general approach to clustering with \econstraints is to jointly cluster the vertices and the \econstraints\!\!, so we include \econstraints as additional vertices of the graph $V$ (i.e. $\Xi \subset V$) and create corresponding decision variables for each of the \econstraints\!\!.
Additionally, we create constraints to ensure that all of the \econstraints are satisfied and none of the \econstraints are incompatible with any of the other members of the predicted clustering.
Given a \econstraint $\xi$ and positive feature $\phi \in \xi^{+}$, let  $\Gamma(\xi, \phi) := \{ v \in V \setminus \Xi: \phi \in \Phi(v) \land \Phi(v) \cap \xi^{-} = \emptyset\}$ be the set of candidate vertices that could satisfy $\phi \in \xi^{+}$.
The resulting integer linear programming problem is as follows
\begin{equation}
    \label{eq:exact_cc_w_econstraints}
    \maximize_{x_u, x_v\in \{e_1, e_2, \ldots, e_n \}} \sum_{(u,v) \in V \times V}\!\! w_{uv} \, x_u^\top x_v 
    \;\; \st \;\; 
    \begin{cases}
     \sum_{v \in \Gamma(\xi, \phi)} \; x_v^\top x_{\xi} \geq 1,   & \textrm {for all } \phi \in \xi^{+} \\
     x_v^\top x_{\xi} = 0, &  \textrm{for all } v \perp \xi
    \end{cases}
\end{equation}
where $w_{uv} \in \RR$ is the pairwise similarity between mentions $u$ and $v$ (usually computed using a trained machine learning model trained on pairs of mentions) and $e_i \in \RR^{n}$ is the $i^{\textrm{th}}$ standard basis vector.
Note that we assume that the $w_{v\xi} = 0$ as in~\cite{pmlr-v162-angell22a}.
The first constraint ensures that all of the positive features in $\xi^+$ are contained in the same cluster as $\xi$ for all $\xi \in \Xi$. The last constraint ensures that $\xi$ is not in the same cluster as anything incompatible with it for all $\xi \in \Xi$.
This problem is intractable to optimize exactly, so we relax (\ref{eq:exact_cc_w_econstraints}) to an SDP in the standard way as follows
\begin{equation}
    \label{eq:sdp_cc_w_econstraints}
    \maximize_{X \succeq 0}  \;\; \langle W, X \rangle \;\; \st \;\;
    \begin{cases}
     X_{vv} = 1, & \textrm{for all } v \in V \\
     X_{uv} \geq 0, & \textrm{for all } u, v \in V \\
     \sum_{v \in \Gamma(\xi, \phi)} \;  X_{v\xi} \geq 1,   & \textrm {for all } \phi \in \xi^{+} \\
     X_{v\xi} = 0, &  \textrm{for all } v \perp \xi
    \end{cases}
\end{equation}
The decision variables  are changed from standard basis vectors, whose dot products determine whether or not two elements (vertices or \econstraints\!\!) are in the same cluster, to a positive semidefinite matrix where each entry represents the global affinity two elements have for one another (corresponding to the row and column of the entry).
We constrain the diagonal of this positive semidefinite matrix to be all ones (representing each element's affinity with itself) and we enforce that all entries in the matrix are non-negative.
The constraints ensuring \econstraint satisfaction are similar to (\ref{eq:exact_cc_w_econstraints}).
Let $W$ be the (weighted) adjacency matrix of the fully connected graph over the $n$ vertices.

\subsubsection{Rounding}
The solution $X_\star$ to~\eqref{eq:sdp_cc_w_econstraints} is a fractional approximation of the problem we wish to solve.
In order to infer a predicted clustering, we need some procedure to round the fractional solution to an integer solution.
\citet{pmlr-v162-angell22a} propose first building a hierarchical clustering over $V$ using $X_\star$ as the similarity function.
Then, a dynamic programming algorithm is used to extract a (predicted) flat clustering from the hierarchical clustering which maximizes the number \econstraints satisfied and the (correlation) clustering objective.
For a more detailed description of the rounding procedure, see~\cite{pmlr-v162-angell22a}. The predicted clustering is then observed by the human users in terms of the features of each predicted cluster enabling the users to provide any additional \econstraints to correct the entity resolution decisions. We simulate this human feedback loop with an oracle \econstraint generator.

\subsubsection{Experimental Setup}
We evaluate the performance of the CGAL and USBS with and without warm-starting on the same three author coreference datasets used in~\cite{pmlr-v162-angell22a}.
In these datasets, each mention is an author-paper pair and the goal is to identify which author-paper pairs were written by the same author.
As standard in author coreference, the datasets are preprocessed into \emph{blocks} (also known as \emph{canopies}) based on the author's first name initial and last name (e.g. authors ``Rajarshi Das'' and ``Ravi Das'' would both be contained in the same block ``\smtt{r das}'', but in a different block than ``Jane Smith'', which would be in ``\smtt{j smith}'').
Within each block, similarity scores are computed between all pairs of mentions using a trained pairwise model~\cite{pmlr-v162-angell22a, subramanian2021s2and}.
We then perform two additional preprocessing steps to convert the many small dense problems into one large sparse problem.
We aggregate all of these pairwise similarity scores into a block diagonal weight matrix and fill the remaining entries with negative one.
We then sparsify this highly structured similarity matrix using a spectral sparsifier~\cite{spielman2008graph}.
\autoref{tab:ecc_dataset_stats} details the dataset statistics including the number of non-zeros in the pairwise similarity matrix after spectral sparsification, denoted nnz($W$).
\begin{table}[t]
\begin{center}
\caption{Author coreference dataset statistics.}
\small
\begin{tabular}{lccccc}
\toprule
 & \# mentions & \# blocks & \# clusters & \textrm{nnz}($W$) & \# features \\
\midrule
\bf PubMed & 315 & 5 & 34 & 3,973 & 14,093\\
\bf QIAN & 410 & 38 & 77 & 5,158 & 10,366\\
 \bf SCAD-zbMATH\qquad & 1,196 & 120 & 166 & 18,608 & 8,203\\
\bottomrule
\end{tabular}
\label{tab:ecc_dataset_stats}
\end{center}
\end{table}

We simulate \econstraint generation using the same oracle implemented in~\cite{pmlr-v162-angell22a}.
The oracle has access to the features of the ground truth clusters and the features of the predicted clusters output by the rounding algorithm and generates a small \econstraint which is satisfied by the ground truth clustering, but is not satisfied by the predicted clustering.
The oracle generates \econstraints one at a time and added to the optimization problem until the ground truth clustering is predicted.
For both solvers, we iterate until the relative error tolerance~\eqref{eq:relative_error_tol} and the following absolute infeasibility condition are satisfied
\begin{equation}
    \|\cA X - \proj_\cK(\cA X)\|_\infty \leq \varepsilon.
\end{equation}
We use $\varepsilon = 10^{-1}$ to determine the stopping condition and find that this is sufficient for the rounding algorithm to satisfy nearly all of the \econstraints generated by the oracle.
We also note that the $\ell^\infty$-norm condition is especially important in this application in order to make sure the relaxed \econstraints are satisfied.
Due to the small problem size, we simplify the implementation by using variants of CGAL and USBS without sketching.

When warm-starting, we predict the values of the expanded primal variable corresponding to the newly added \econstraint by computing representations of each vertex using a low-rank factorization, performing a weighted average of the representations of the positive vertices in the newly added \econstraint\!\!, and expanding the factorization into an initialization of primal variable for the new problem.
We pad the dual variable with zeros.
Since there is no pairwise similarity between \econstraints and mentions, we find that scaling up the norms of $A_i$'s corresponding to the constraints $X_{vv} = 1$ for all $v$ such that $\Phi(v) \cap \xi_\textrm{new} \ne \emptyset$, where $\xi_\textrm{new}$ is the newly added \econstraint\!\!.
We use the following USBS hyperparameters in our experiments: $\rho = 0.01$, $\beta = 0.25$, $k_c = 3$, $k_p = 0$.
See the implementation for more details.

\section{Additional Experimental Results}
\label{sec:add_expt_results}

\subsection{MaxCut}

\autoref{fig:maxcut_plot_144} plots three different convergence measures over time for a large data instance, $\smsf{144}$, for a total time of 12 hours. For this instance, warm-starting helps achieve a faster convergence rate of both relative primal objective suboptimality and relative infeasibility. Additionally, the weight of the cut produced from the primal iterates generated by USBS drops in value less when warm-starting as compared with CGAL. This is an additional indication that USBS is able to utilize the warm-start initialization better than CGAL. Finally, observe that USBS, with and without warm-starting, produces a marginally better graph cut than CGAL.
\begin{figure*}[h!]
\captionsetup[subfigure]{labelformat=empty}
\subfloat[]{%
  \includegraphics[width=0.32\textwidth]{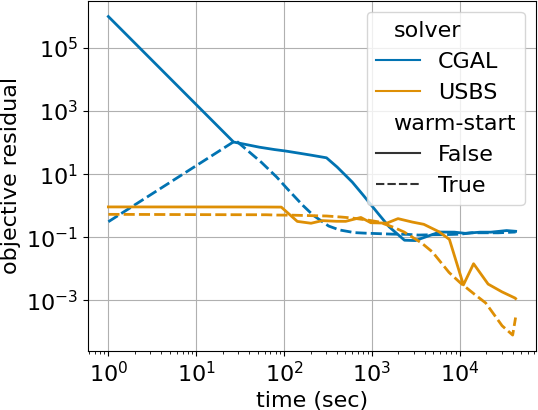}
}
\hfill
\subfloat[]{%
  \includegraphics[width=0.32\textwidth]{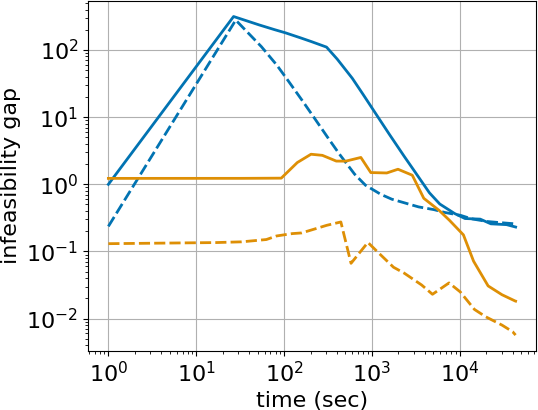}
}
\hfill
\subfloat[]{%
  \includegraphics[width=0.32\textwidth]{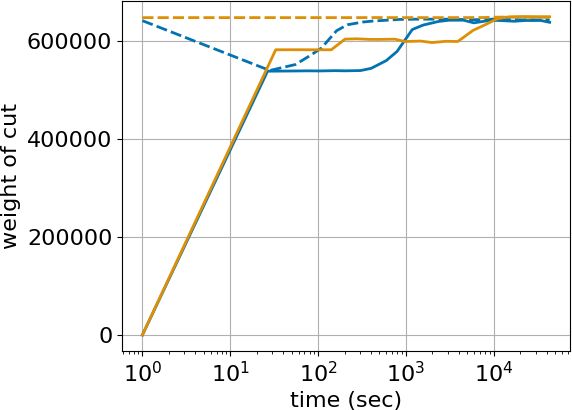}
}
\vskip -0.2in
\caption{\textbf{Convergence measures on instance} \smsf{144}\textbf{.}
We solve instance \smsf{144} from DIMACS10 and plot the primal objective sub-optimality (objective residual, $\downarrow$), relative infeasibility (infeasibility gap, $\downarrow$), and weight of the cut ($\uparrow$) produced by the rounding procedure. The warm-started runs use 99\% of the original data to obtain a warm-start initialization. We observe that USBS is able to more reliably leverage a warm-start initialization. In these plots, USBS is executed with $k_c = 8, k_p = 8$. All runs were executed on a compute node with 16 cores and 128GB of RAM.}
\label{fig:maxcut_plot_144}
\vskip -0.1in
\end{figure*}
\begin{figure*}[h!]
\subfloat[$k_p = 0$\label{subfig:infeas-cpu-0}]{%
  \includegraphics[width=0.32\textwidth]{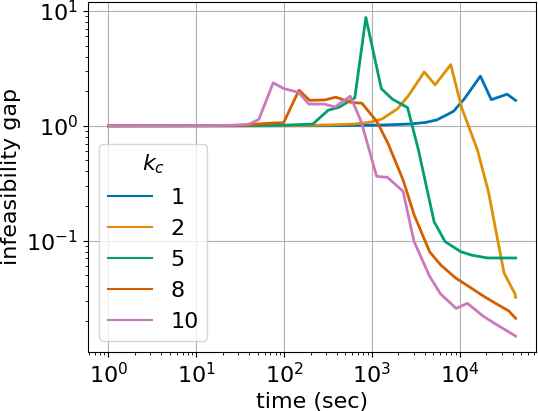}
}
\hfill
\subfloat[$k_p = 1$\label{subfig:infeas-cpu-1}]{%
  \includegraphics[width=0.32\textwidth]{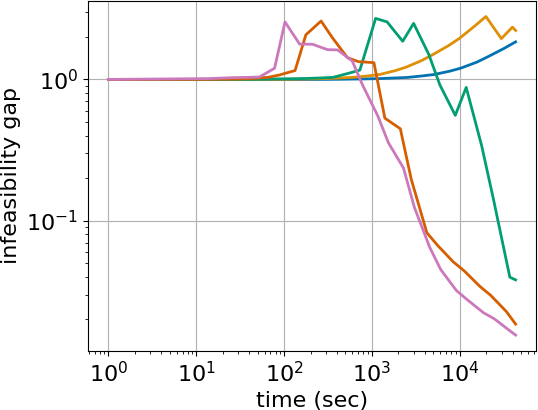}
}
\hfill
\subfloat[$k_p = 2$\label{subfig:infeas-cpu-2}]{%
  \includegraphics[width=0.32\textwidth]{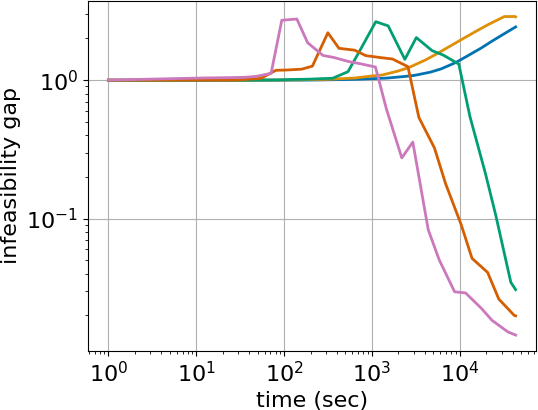}
}

\subfloat[$k_p = 5$\label{subfig:infeas-cpu-5}]{%
  \includegraphics[width=0.32\textwidth]{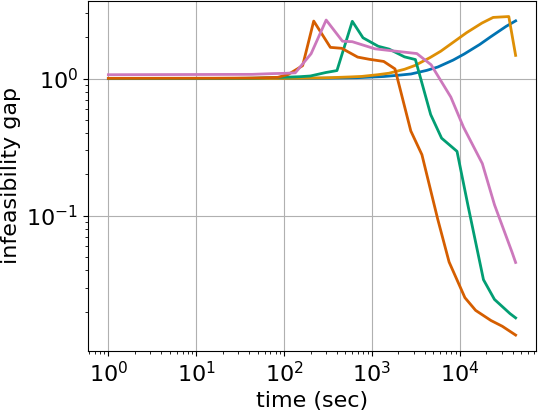}
}
\hfill
\subfloat[$k_p = 8$\label{subfig:infeas-cpu-8}]{%
  \includegraphics[width=0.32\textwidth]{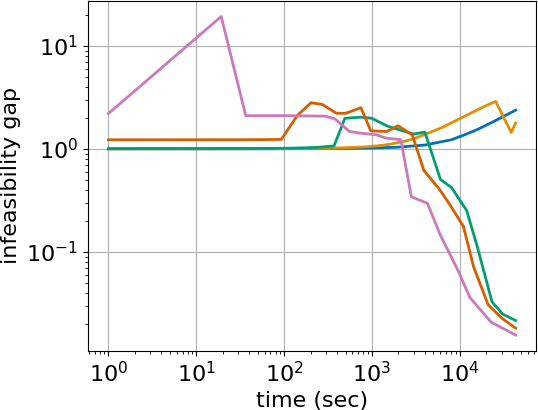}
}
\hfill
\subfloat[$k_p = 10$\label{subfig:infeas-cpu-10}]{%
  \includegraphics[width=0.32\textwidth]{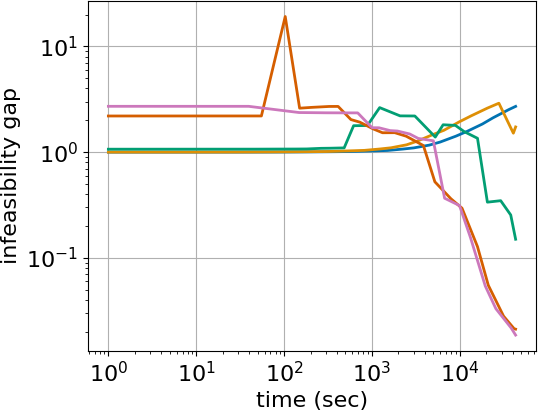}
}
\caption{
\textbf{Infeasibility gap vs. time for different settings of $k_c$ and $k_p$.} In each plot, the x-axis is time (up to 12 hours) and the y-axis is the relative infeasibility gap. In every case, USBS is cold-started on the \smsf{144} instance from the \texttt{DIMACS10} dataset. Each plot considers one value of $k_p$ and several values of $k_c$. All runs were executed on a compute node with 16 cores and 128GB of RAM. We observe that USBS performs best when $k_c \geq k_p$.} 
\label{fig:infeas_grid_cpu}
\end{figure*}
\begin{figure*}[h!]
\subfloat[$k_p = 0$\label{subfig:infeas-gpu-0}]{%
  \includegraphics[width=0.32\textwidth]{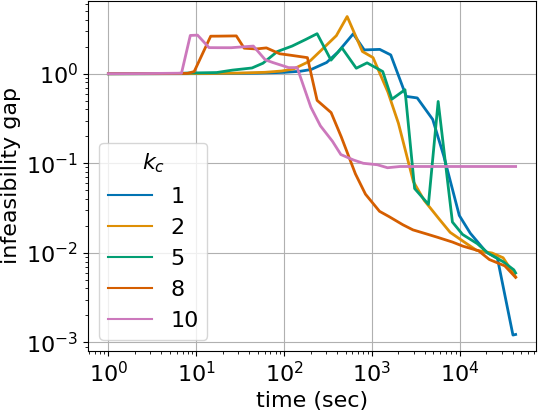}
}
\hfill
\subfloat[$k_p = 1$\label{subfig:infeas-gpu-1}]{%
  \includegraphics[width=0.32\textwidth]{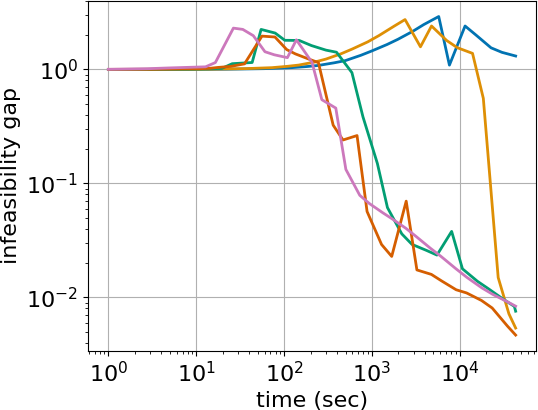}
}
\hfill
\subfloat[$k_p = 2$\label{subfig:infeas-gpu-2}]{%
  \includegraphics[width=0.32\textwidth]{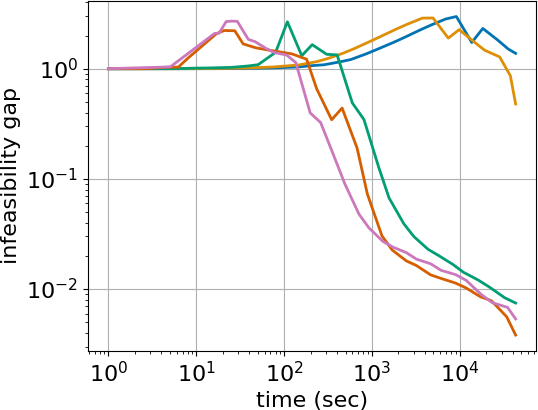}
}

\subfloat[$k_p = 5$\label{subfig:infeas-gpu-5}]{%
  \includegraphics[width=0.32\textwidth]{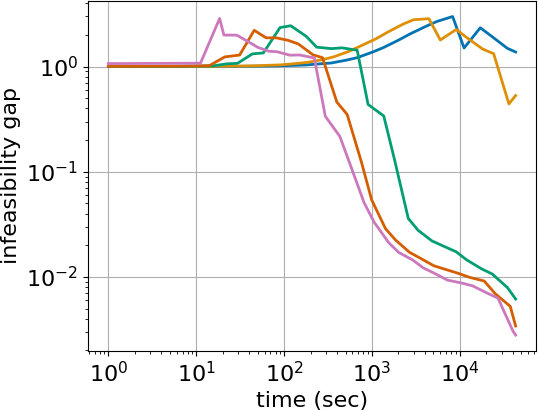}
}
\hfill
\subfloat[$k_p = 8$\label{subfig:infeas-gpu-8}]{%
  \includegraphics[width=0.32\textwidth]{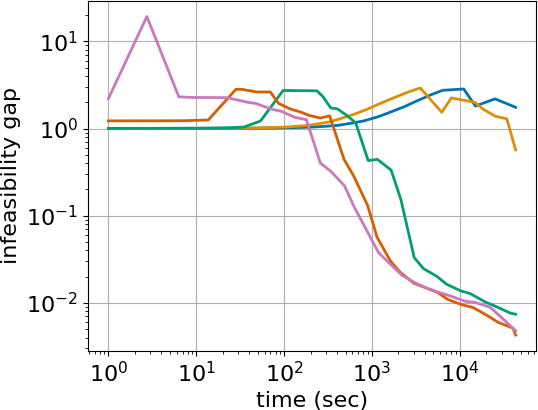}
}
\hfill
\subfloat[$k_p = 10$\label{subfig:infeas-gpu-10}]{%
  \includegraphics[width=0.32\textwidth]{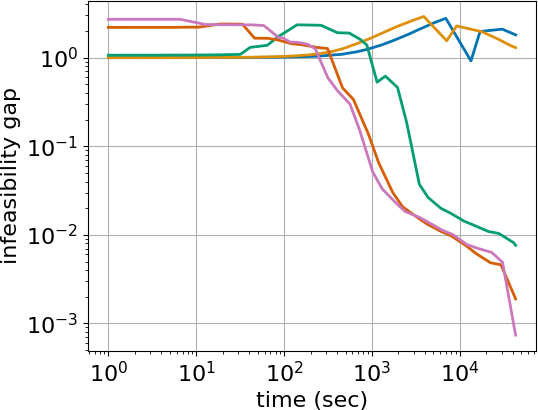}
}
\caption{\textbf{Infeasibility gap vs. time for different settings of $k_c$ and $k_p$.} In each plot, the x-axis is time (up to 12 hours) and the y-axis is the relative infeasibility gap. In every case, USBS is cold-started on the \smsf{144} instance from the \texttt{DIMACS10} dataset. Each plot considers one value of $k_p$ and several values of $k_c$. All runs were executed on a single NVIDIA GeForce 1080 Ti GPU. We observe that USBS performs best when $k_c \geq k_p$.} 
\label{fig:infeas_grid_gpu}
\end{figure*}
\begin{figure*}[h!]
\subfloat[$k_p = 0$\label{subfig:obj-cpu-0}]{%
  \includegraphics[width=0.32\textwidth]{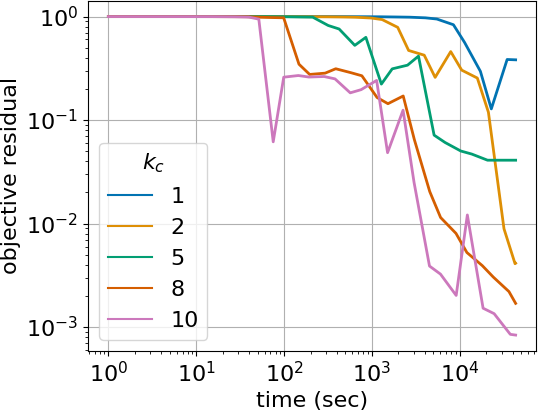}
}
\hfill
\subfloat[$k_p = 1$\label{subfig:obj-cpu-1}]{%
  \includegraphics[width=0.32\textwidth]{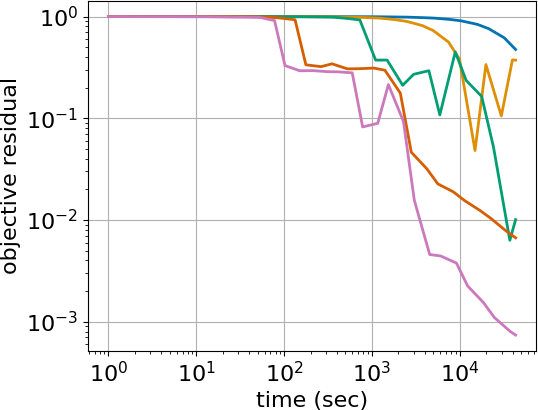}
}
\hfill
\subfloat[$k_p = 2$\label{subfig:obj-cpu-2}]{%
  \includegraphics[width=0.32\textwidth]{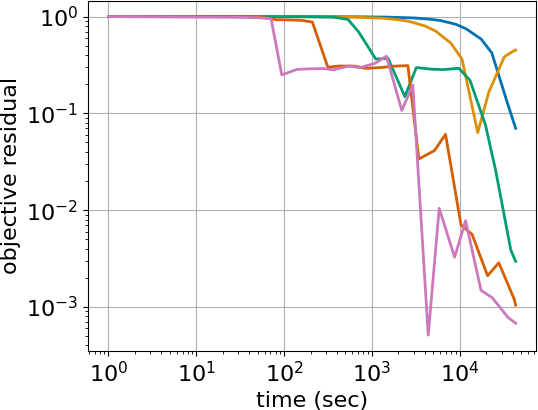}
}

\subfloat[$k_p = 5$\label{subfig:obj-cpu-5}]{%
  \includegraphics[width=0.32\textwidth]{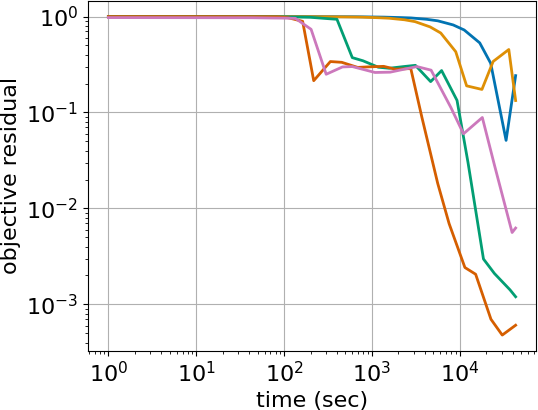}
}
\hfill
\subfloat[$k_p = 8$\label{subfig:obj-cpu-8}]{%
  \includegraphics[width=0.32\textwidth]{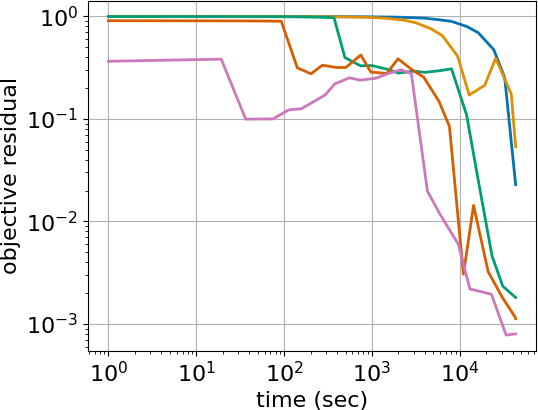}
}
\hfill
\subfloat[$k_p = 10$\label{subfig:obj-cpu-10}]{%
  \includegraphics[width=0.32\textwidth]{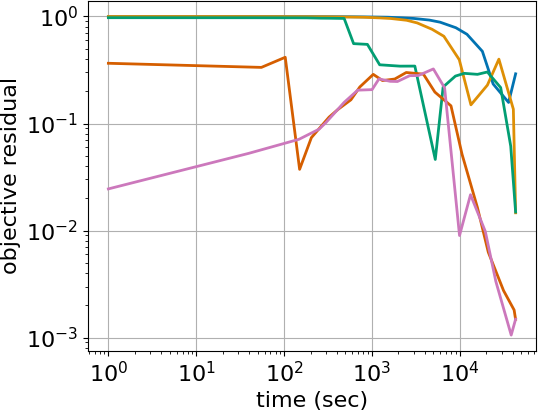}
}
\caption{\textbf{Objective gap vs. time for different settings of $k_c$ and $k_p$.} In each plot, the x-axis is time (up to 12 hours) and the y-axis is the primal objective suboptimality. In every case, USBS is cold-started on the \smsf{144} instance from the \texttt{DIMACS10} dataset. Each plot considers one value of $k_p$ and several values of $k_c$. All runs were executed on a compute node with 16 cores and 128GB of RAM. We observe that USBS performs best when $k_c \geq k_p$.}
\label{fig:obj_grid_cpu}
\end{figure*}
\begin{figure*}[h!]
\subfloat[$k_p = 0$\label{subfig:obj-gpu-0}]{%
  \includegraphics[width=0.32\textwidth]{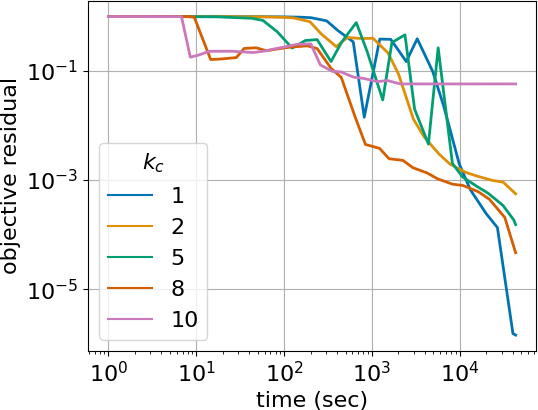}
}
\hfill
\subfloat[$k_p = 1$\label{subfig:obj-gpu-1}]{%
  \includegraphics[width=0.32\textwidth]{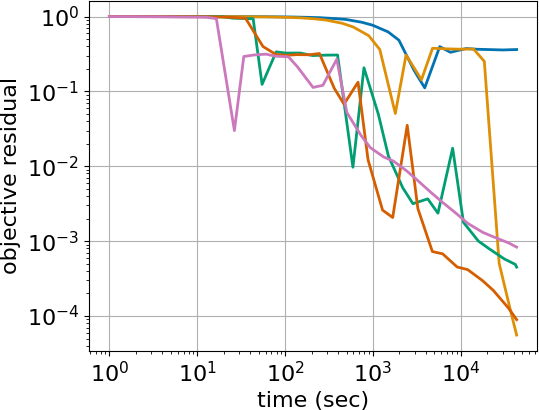}
}
\hfill
\subfloat[$k_p = 2$\label{subfig:obj-gpu-2}]{%
  \includegraphics[width=0.32\textwidth]{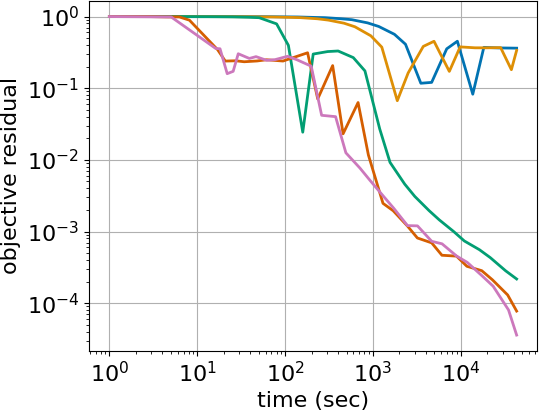}
}

\subfloat[$k_p = 5$\label{subfig:obj-gpu-5}]{%
  \includegraphics[width=0.32\textwidth]{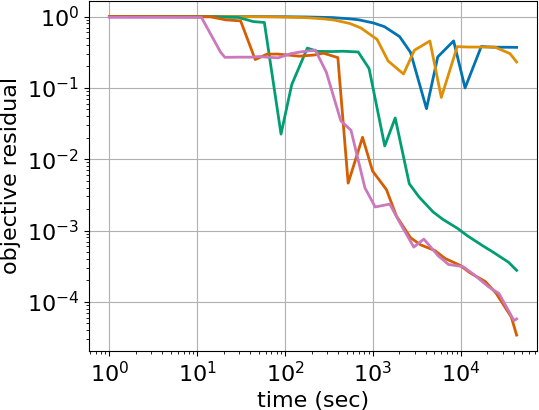}
}
\hfill
\subfloat[$k_p = 8$\label{subfig:obj-gpu-8}]{%
  \includegraphics[width=0.32\textwidth]{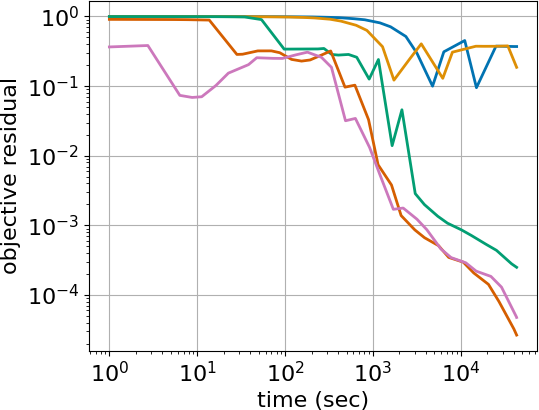}
}
\hfill
\subfloat[$k_p = 10$\label{subfig:obj-gpu-10}]{%
  \includegraphics[width=0.32\textwidth]{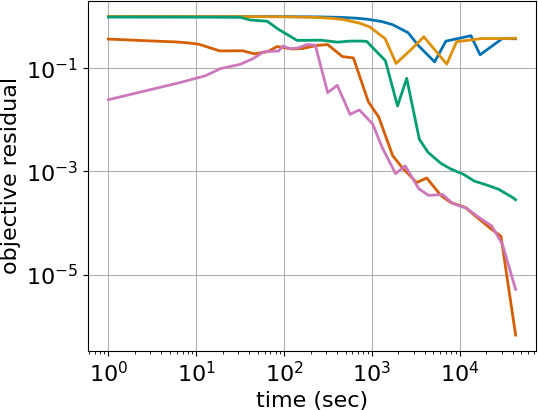}
}
\caption{\textbf{Objective gap vs. time for different settings of $k_c$ and $k_p$.} In each plot, the x-axis is time (up to 12 hours) and the y-axis is the relative primal objective suboptimality. In every case, USBS is cold-started on the \smsf{144} instance from the \texttt{DIMACS10} dataset. Each plot considers one value of $k_p$ and several values of $k_c$.  All runs were executed on a single NVIDIA GeForce 1080 Ti GPU. We observe that USBS performs best when $k_c \geq k_p$.} 
\label{fig:obj_grid_gpu}
\end{figure*}
\begin{figure*}[h!]
\subfloat[$k_c = 8, k_p = 0$\label{subfig:infeas-hardware-0}]{%
  \includegraphics[width=0.32\textwidth]{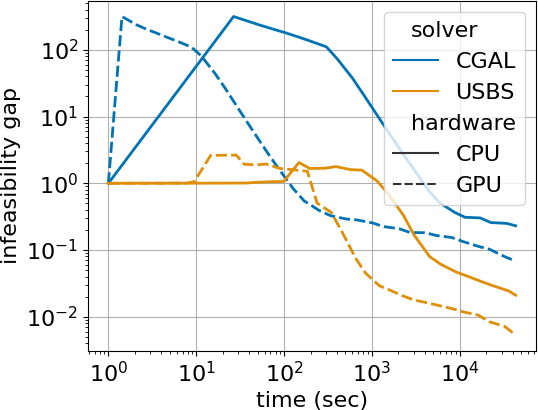}
}
\hfill
\subfloat[$k_c = 8, k_p = 1$\label{subfig:infeas-hardware-1}]{%
  \includegraphics[width=0.32\textwidth]{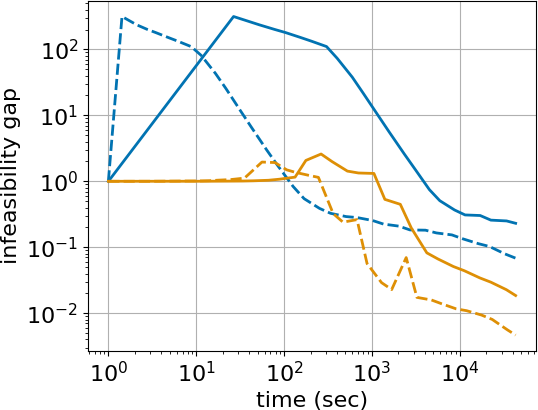}
}
\hfill
\subfloat[$k_c = 8, k_p = 2$\label{subfig:infeas-hardware-2}]{%
  \includegraphics[width=0.32\textwidth]{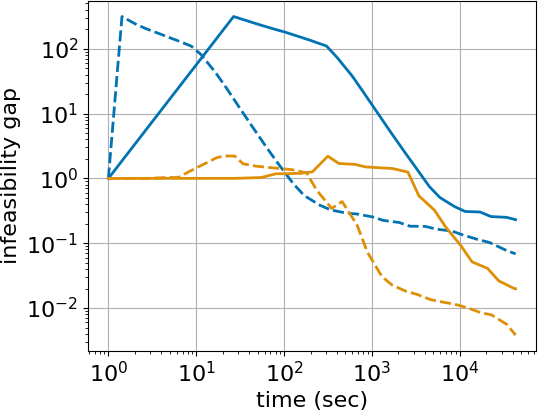}
}

\subfloat[$k_c = 8, k_p = 5$\label{subfig:infeas-hardware-5}]{%
  \includegraphics[width=0.32\textwidth]{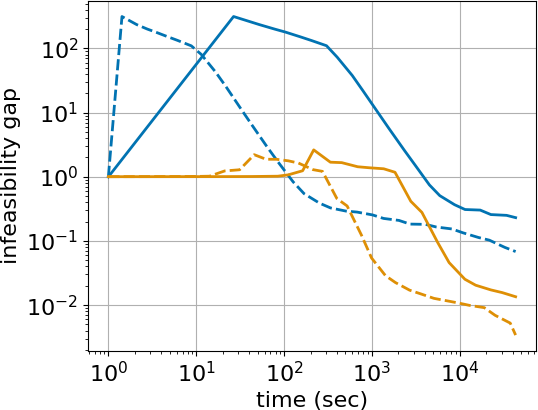}
}
\hfill
\subfloat[$k_c = 8, k_p = 8$\label{subfig:infeas-hardware-8}]{%
  \includegraphics[width=0.32\textwidth]{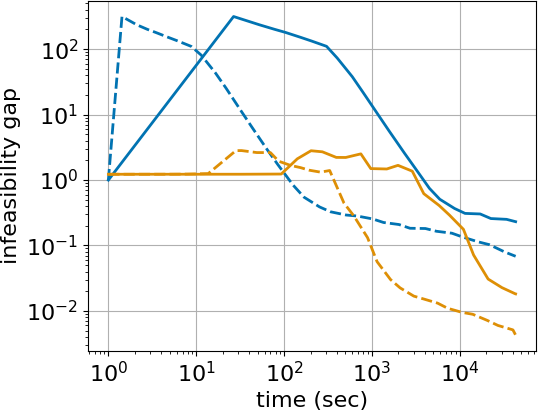}
}
\hfill
\subfloat[$k_c = 8, k_p = 10$\label{subfig:infeas-hardware-10}]{%
  \includegraphics[width=0.32\textwidth]{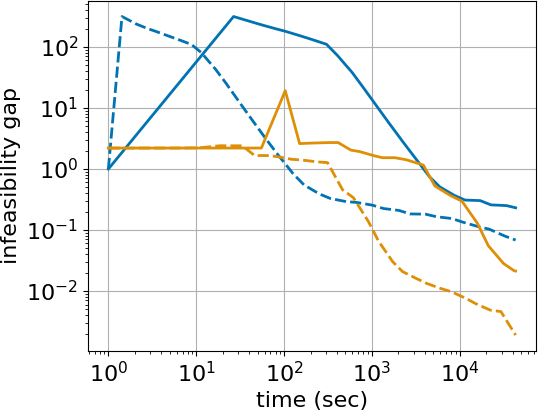}
}
\caption{\textbf{Infeasibility gap vs. time on CPU and GPU.} In each plot, the x-axis is time (up to 12 hours) and the y-axis is the relative infeasibility gap. In every setting of $k_c$ and $k_p$, we compare CGAL and USBS  cold-started on the \smsf{144} instance from the \texttt{DIMACS10} dataset on both a compute node with 16 cores and 128GB of RAM (\texttt{CPU}) and a single NVIDIA GeForce 1080 Ti GPU (\texttt{GPU}). We observe that USBS performs best when $k_c \geq k_p$.} 
\label{fig:infeas_cpu_vs_gpu8}
\end{figure*}
\begin{figure*}[h!]
\subfloat[$k_c = 8, k_p = 0$\label{subfig:obj-hardware-0}]{%
  \includegraphics[width=0.32\textwidth]{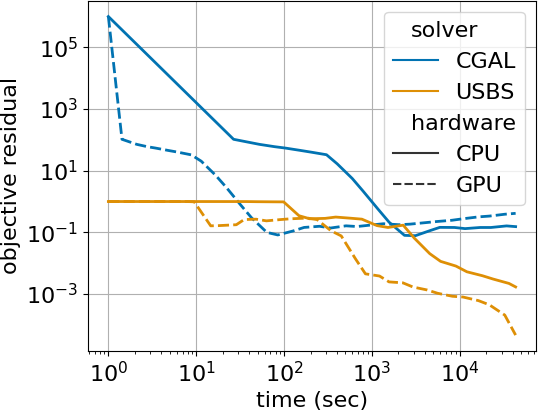}
}
\hfill
\subfloat[$k_c = 8, k_p = 1$\label{subfig:obj-hardware-1}]{%
  \includegraphics[width=0.32\textwidth]{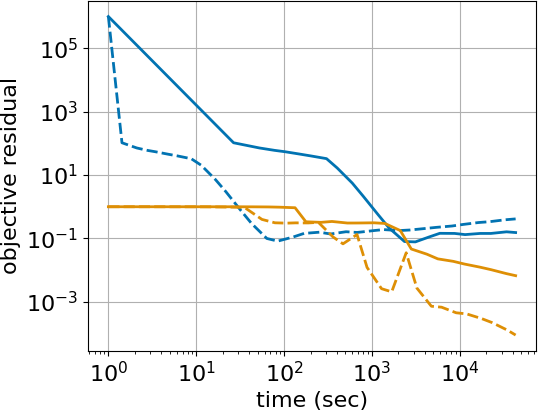}
}
\hfill
\subfloat[$k_c = 8, k_p = 2$\label{subfig:obj-hardware-2}]{%
  \includegraphics[width=0.32\textwidth]{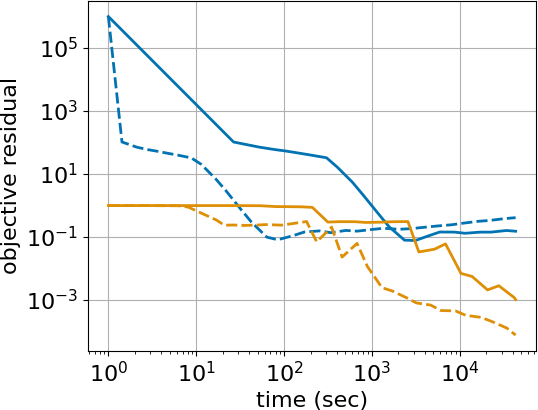}
}

\subfloat[$k_c = 8, k_p = 5$\label{subfig:obj-hardware-5}]{%
  \includegraphics[width=0.32\textwidth]{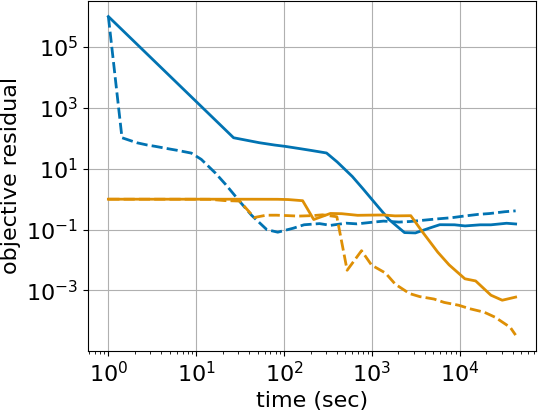}
}
\hfill
\subfloat[$k_c = 8, k_p = 8$\label{subfig:obj-hardware-8}]{%
  \includegraphics[width=0.32\textwidth]{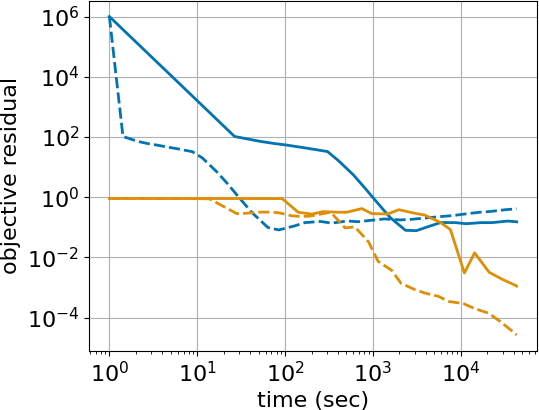}
}
\hfill
\subfloat[$k_c = 8, k_p = 10$\label{subfig:obj-hardware-10}]{%
  \includegraphics[width=0.32\textwidth]{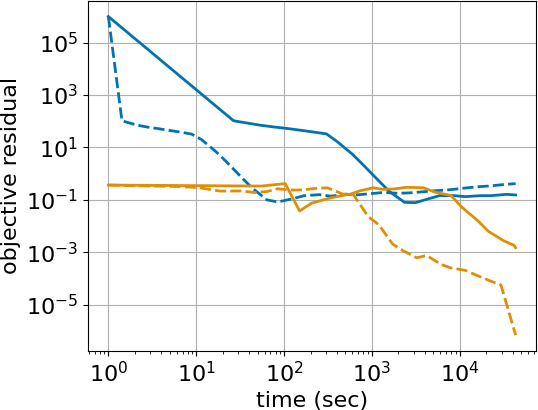}
}
\caption{\textbf{Objective gap vs. time on CPU and GPU.} In each plot, the x-axis is time (up to 12 hours) and the y-axis is the relative primal objective suboptimality. In every setting of $k_c$ and $k_p$, we compare CGAL and USBS cold-started on the \smsf{144} instance from the \texttt{DIMACS10} dataset on both a compute node with 16 cores and 128GB of RAM (\texttt{CPU}) and a single NVIDIA GeForce 1080 Ti GPU (\texttt{GPU}). We observe that USBS performs best when $k_c \geq k_p$.} 
\label{fig:obj_cpu_vs_gpu8}
\end{figure*}

\cref{fig:infeas_grid_cpu,fig:infeas_grid_gpu} compare USBS's infeasibility gap against time for different settings of $k_c$ and $k_p$ on the \smsf{144} instance from \texttt{DIMACS10} on both CPU and GPU. 
\cref{fig:obj_grid_cpu,fig:obj_grid_gpu} compare USBS's infeasibility gap against time for different settings of $k_c$ and $k_p$ on the \smsf{144} instance from \texttt{DIMACS10} on both CPU and GPU. 
\cref{fig:infeas_cpu_vs_gpu8,fig:obj_cpu_vs_gpu8} compare CGAL and USBS on CPU and GPU for different settings of $k_c$ and $k_p$ on the \smsf{144} instance from \texttt{DIMACS10}.
We observe anywhere from 10-25x speedup on GPU as compared to CPU. We also observe that USBS performs best when $k_c \geq k_p$.

\clearpage
\subsection{QAP}
\autoref{fig:qap_relative_gap} and~\autoref{fig:qap_best_relative_gap} plot \smtt{relative gap} and \smtt{best relative gap} against time (in seconds) for three instances, respectively. 
\autoref{fig:qap_relative_gap} shows that even for large instances the \smtt{best relative gap} is obtained within the first few minutes of the hour of optimization. 
In these instances, we can also observe that USBS not only obtains a better \smtt{best relative gap}, but also is able to leverage a warm-start solution to improve performance.
\autoref{fig:qap_bar_chart} shows the \smtt{best relative gap} obtained for ten data instances over one hour of optimization. It can be seen that warm-starting does not always yield a better \smtt{best relative gap}, but USBS is better able to leverage a warm-start solution than CGAL. 
These results show that warm-starting is a beneficial heuristic for finding a quality approximate solution to QAPs. 
The warm-starting technique used in these experiments is not the only possible warm-starting initialization strategy.
For examples, one could generate several warm-start initializations for~\eqref{eq:qap_sdp} by dropping other rows and columns of $D$ and $W$ besides the last row and column.
Then, after creating \smtt{n} warm-start initializations, we could take the best upper bound obtained by optimizing~\eqref{eq:qap_sdp} from each of those warm-start initializations.
\begin{figure*}[h!]
\subfloat[pr136\label{subfig:gap_pr136}]{%
  \includegraphics[width=0.31\textwidth]{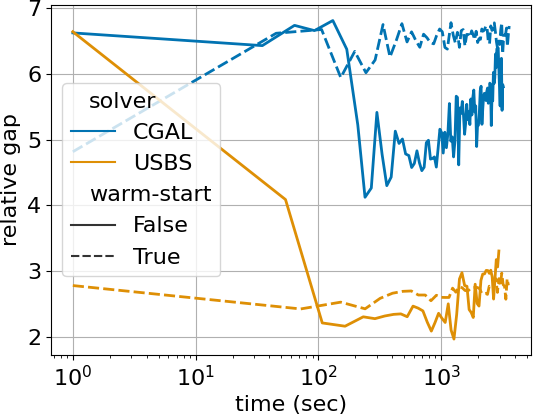}
}
\hfill
\subfloat[kroA150\label{subfig:gap_kroA150}]{%
  \includegraphics[width=0.31\textwidth]{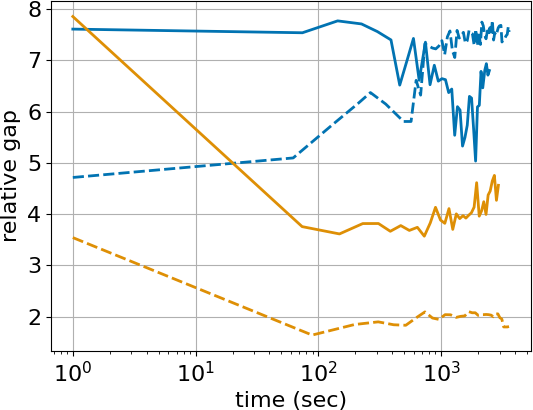}
}
\hfill
\subfloat[tai150b\label{subfig:gap_tai150b}]{%
  \includegraphics[width=0.32\textwidth]{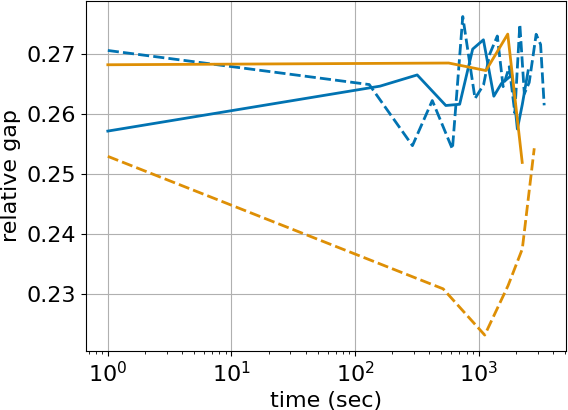}
}
\caption{\textbf{\smtt{relative gap} ($\downarrow$) vs.~time.} We plot the \texttt{relative gap} (y-axis) against time in seconds (x-axis) for three instances from \texttt{QAPLIB} and \texttt{TSPLIB}. We observe that for both algorithms the best rounded solution is found early in optimization. In addition, we observe that USBS is able to more reliably leverage a warm-start initialization.}
\label{fig:qap_relative_gap}
\vskip -0.1in
\end{figure*}
\begin{figure*}[h!]
\subfloat[pr136\label{subfig:gap_pr136}]{%
  \includegraphics[width=0.31\textwidth]{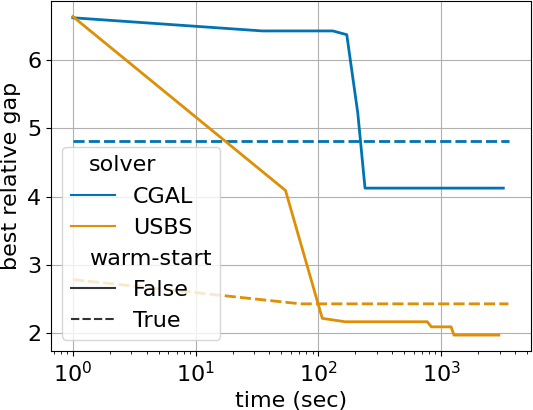}
}
\hfill
\subfloat[kroA150\label{subfig:gap_kroA150}]{%
  \includegraphics[width=0.31\textwidth]{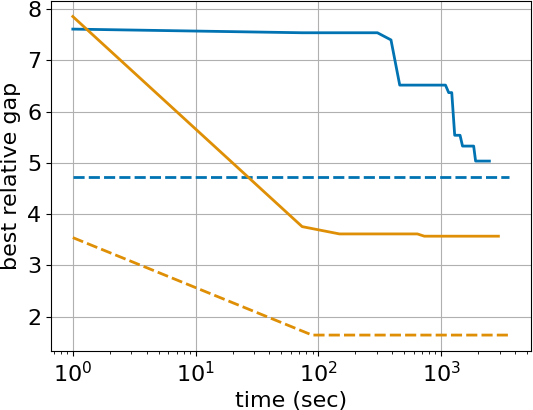}
}
\hfill
\subfloat[tai150b\label{subfig:gap_tai150b}]{%
  \includegraphics[width=0.32\textwidth]{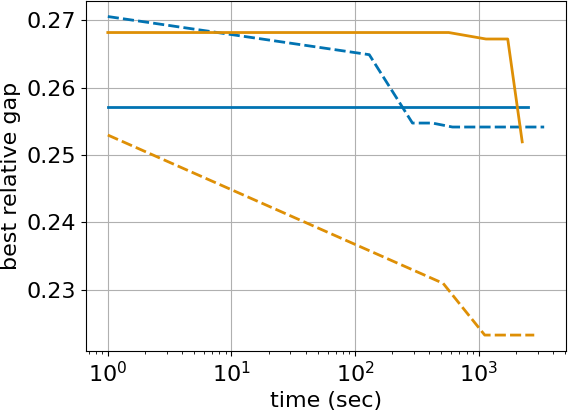}
}
\caption{\textbf{\smtt{best relative gap} ($\downarrow$) vs.~time.} We plot the \smtt{best relative gap} (y-axis) against time in seconds (x-axis) for three instances from \texttt{QAPLIB} and \texttt{TSPLIB}. We observe that USBS is able to more reliably leverage a warm-start initialization.}
\label{fig:qap_best_relative_gap}
\vskip -0.1in
\end{figure*}
\begin{figure*}[h!]
\includegraphics[width=\textwidth]{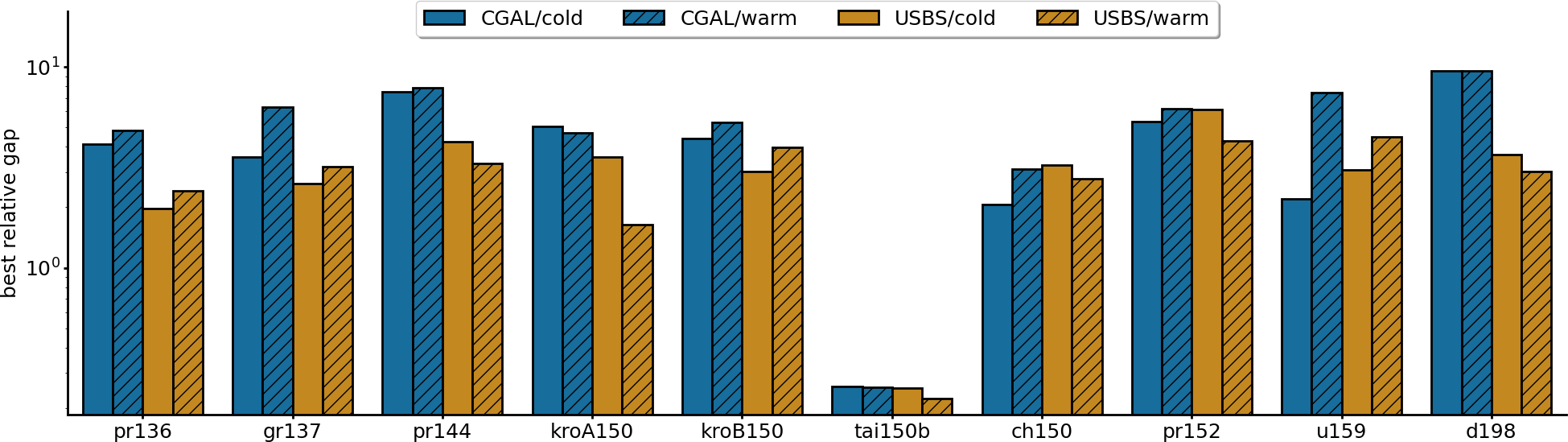}
\caption{\textbf{\smtt{best relative gap ($\downarrow$).}} The \smtt{best relative gap} obtained in one hour of optimization is shown for ten instances from \texttt{QAPLIB} and \texttt{TSPLIB}. We observe that on most problem instance that USBS produces a better relative gap than CGAL and that a warm-start initialization helps USBS obtain a better relative gap. CGAL is much less reliable in leveraging a warm-start initialization.}
\label{fig:qap_bar_chart}
\end{figure*}

\clearpage
\subsection{Interactive Entity Resolution with \econstraints}
\begin{wrapfigure}{r}{8.5cm}
\includegraphics[width=8.5cm]{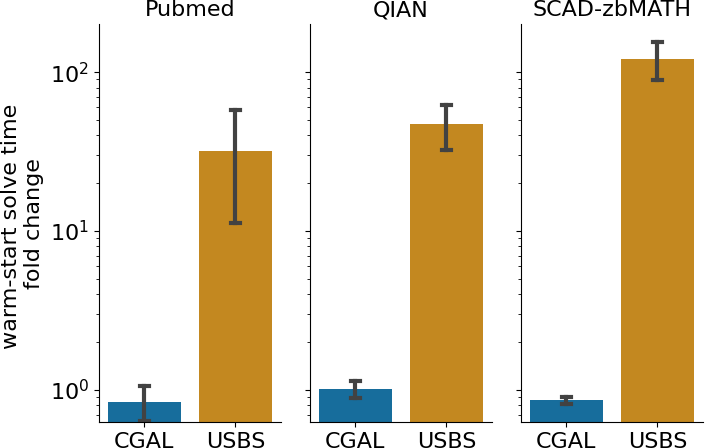}
\caption{\textbf{Average warm-start fold change ($\uparrow$).} Warm-start fold change represents the ratio of SDP solve time without warm-starting divided by the SDP solve time with warm-starting per \econstraint\!\!. The error bars indicate one standard deviation from the mean. We observe that USBS is able to much more reliably leverage a warm-start initialization. In addition, we observe that the performance gap between CGAL and USBS grows as the problem size grows.}
\label{fig:ecc_relative_improvement}
\end{wrapfigure}
\autoref{fig:ecc_relative_improvement} shows the average warm-start SDP solve time fold change. Warm-start fold change greater than one indicates a speedup in solve time while warm-start fold change less than one indicates a slowdown in solve time. We see that the warm-start fold change for CGAL is less than or equal to one, indicating a slowdown consistent with~\autoref{fig:ecc_cumulative_solve_time}. We see that on average warm-starting affords USBS a 20-100x speedup on average per \econstraint\!\!. Note that this is much faster than the 2-3x we see in~\autoref{fig:ecc_cumulative_solve_time}. This is due to the fact that sometimes warm-starting does not help USBS, which makes intuitive sense if the warm-start initialization is far away from the solution set for the new SDP. We believe that with some further experimentation including additional pairwise learned similarities between mentions and \econstraints or different warm-starting strategies we might be able to mitigate these situations. 
Regardless, in most cases, warm-starting provides a significant improvement in convergence time when using USBS.
We also note that~\autoref{fig:ecc_relative_improvement} shows that warm-starting provides more of a benefit the larger the SDP we are trying to solve.

\end{document}